\documentclass[a4paper,12pt]{amsart}
\usepackage[english]{babel}
\usepackage[T1]{fontenc}
\usepackage[left=1in,right=1in,top=1.2in,bottom=1.2in]{geometry}
\usepackage{times}
\usepackage{microtype}

\usepackage{commath,amssymb,amscd,mathrsfs,mathtools,dsfont,bbm,multirow,array,caption,comment,pifont}
\usepackage[all]{xy}
\usepackage{enumitem}
\usepackage{longtable}

\usepackage{cite}			%won't show labels for citation

\allowdisplaybreaks

\usepackage[usenames,dvipsnames,svgnames,table]{xcolor}
\definecolor{red}{RGB}{255,25,25}
\definecolor{blue}{RGB}{25,50,200}
\usepackage{tikz-cd}
\usepackage[pagebackref,linktocpage]{hyperref}

\hypersetup{
pdftitle={},				% title
pdfauthor={Fei Hu},	% author
colorlinks=true,			% false: boxed links; true: colored links
linkcolor=red,			% color of internal links (change box color with linkbordercolor)
citecolor=MidnightBlue,	% color of links to bibliography
filecolor=magenta,		% color of file links
urlcolor=MidnightBlue	% color of external links
}

\usepackage{cleveref}	% after hyperref!!

\newtheorem{theorem}{Theorem}[section]
\crefname{theorem}{Theorem}{Theorems}
\newtheorem{lemma}[theorem]{Lemma}
\crefname{lemma}{Lemma}{Lemmas}
\newtheorem{proposition}[theorem]{Proposition}
\crefname{proposition}{Proposition}{Propositions}

\crefname{prop}{Proposition}{Propositions}
\newtheorem{corollary}[theorem]{Corollary}
\crefname{corollary}{Corollary}{Corollaries}

\crefname{cor}{Corollary}{Corollaries}
\newtheorem{conjecture}[theorem]{Conjecture}
\crefname{conjecture}{Conjecture}{Conjectures}

\crefname{conj}{Conjecture}{Conjectures}
\newtheorem*{conj*}{Conjecture}
\crefname{conj}{Conjecture}{Conjectures}

\crefname{conjA}{Conjecture}{Conjecture}

\crefname{conjB}{Conjecture}{Conjecture}

\crefname{conjC}{Conjecture}{Conjecture}

\crefname{conjDk}{Conjecture}{Conjecture}

\crefname{conjD}{Conjecture}{Conjecture}

\crefname{conjH}{Conjecture}{Conjecture}

\crefname{conjGr}{Conjecture}{Conjecture}

\theoremstyle{definition}
\newtheorem{definition}[theorem]{Definition}
\crefname{definition}{Definition}{Definitions}

\crefname{defn}{Definition}{Definitions}

\crefname{example}{Example}{Examples}

\crefname{notation}{Notation}{Notation}
\newtheorem*{notation*}{Notation}
\crefname{notation}{Notation}{Notation}
\newtheorem*{convention*}{Convention}
\crefname{convention}{Convention}{Convention}

\crefname{problem}{Problem}{Problems}
\newtheorem{question}[theorem]{Question}
\crefname{question}{Question}{Questions}

\crefname{condition}{Condition}{Conditions}

\crefname{assumption}{Assumption}{Assumptions}
\newtheorem{hypothesis}[theorem]{Hypothesis}
\crefname{hypothesis}{Hypothesis}{hypotheses}

\theoremstyle{remark}

\crefname{rmk}{Remark}{Remarks}
\newtheorem*{rmk*}{Remark}
\crefname{rmk}{Remark}{Remarks}
\newtheorem{remark}[theorem]{Remark}
\crefname{remark}{Remark}{Remarks}

\crefname{fact}{Fact}{Facts}
\newtheorem{claim}[theorem]{Claim}
\crefname{claim}{Claim}{Claims}
\newtheorem*{claim*}{Claim}
\crefname{claim}{Claim}{Claims}

\crefname{step}{Step}{Steps}

\crefname{case}{Case}{Cases}

\numberwithin{equation}{section}
\renewcommand{\emptyset}{\varnothing}

\newcommand{\injmap}{\hookrightarrow}

\newcommand{\doi}[1]{\href{https://doi.org/#1}{{\tt doi:#1}}}

\newcommand{\bB}{\mathbf{B}}
\newcommand{\bC}{\mathbf{C}}

\newcommand{\bK}{\mathbf{K}}

\newcommand{\bN}{\mathbf{N}}

\newcommand{\bP}{\mathbf{P}}
\newcommand{\bQ}{\mathbf{Q}}
\newcommand{\bR}{\mathbf{R}}

\newcommand{\bZ}{\mathbf{Z}}

\newcommand{\cF}{\mathcal{F}}
\newcommand{\cG}{\mathcal{G}}

\newcommand{\cO}{\mathcal{O}}

\newcommand{\sO}{\mathscr{O}}

\newcommand{\GL}{\mathrm{GL}}

\newcommand{\Alb}{\operatorname{Alb}}

\newcommand{\Amp}{\mathsf{Amp}}

\newcommand{\Aut}{\operatorname{Aut}}

\newcommand{\Bs}{\operatorname{Bs}}
\newcommand{\CH}{\mathsf{CH}}

\newcommand{\End}{\operatorname{End}}

\newcommand{\id}{\operatorname{id}}

\newcommand{\Mat}{\operatorname{M}}

\newcommand{\Nef}{\mathsf{Nef}}
\newcommand{\N}{\mathsf{N}}

\newcommand{\NS}{\mathsf{NS}}

\newcommand{\num}{\equiv}
\newcommand{\wnum}{\equiv_{\mathsf{w}}}
\newcommand{\pic}{\mathbf{Pic}}

\newcommand{\Pic}{\mathsf{Pic}}

\newcommand{\Qbar}{\overline{\mathbf{Q}}}
\newcommand{\rank}{\operatorname{rank}}

\begin{document}

\title[Canonical heights for abelian group actions]{Canonical heights for abelian group actions of maximal dynamical rank}

\author{Fei Hu}
\address{Department of Mathematics, Nanjing University, Nanjing, China}
\email{\href{mailto:fhu@nju.edu.cn}{\tt fhu@nju.edu.cn}}

\author{Guolei Zhong}
\address{Center for Complex Geometry,
	Institute for Basic Science (IBS),
	55 Expo-ro, Yuseong-gu, Daejeon, 34126, Republic of Korea}
\email{\href{mailto:guolei@ibs.re.kr}{\tt guolei@ibs.re.kr}}
\email{\href{mailto:zhongguolei@u.nus.edu}{\tt zhongguolei@u.nus.edu}}

\begin{abstract}
Let $X$ be a smooth projective variety of dimension $n\geq 2$ and $G\cong\mathbf{Z}^{n-1}$ a free abelian group of automorphisms of $X$ over $\overline{\mathbf{Q}}$.
Suppose that $G$ is of positive entropy.
We construct a canonical height function $\widehat{h}_G$ associated with $G$, corresponding to a nef and big $\mathbf{R}$-divisor, satisfying the Northcott property.
By characterizing its null locus, we prove the Kawaguchi--Silverman conjecture for each element of \(G\). 
As another application, we determine the height counting function for non-periodic points.
\end{abstract}



\thanks{The first author was supported by a grant from Nanjing University (No.~14912209) and a grant from the National Natural Science Foundation of China (No.~13001128). The second author was supported by the Institute for Basic Science (IBS-R032-D1-2023-a00).}

\dedicatory{In memory of Professor Nessim Sibony (1947--2021)}

\date{}

\maketitle

\section{Introduction}
\label{section:intro}

Given a surjective holomorphic self-map \(f\) of a compact K\"ahler manifold \(M\) of dimension \(n\geq 2\), the topological entropy \(h_{\textrm{top}}(f)\) of \(f\) is a key dynamical invariant to measure the divergence of the orbits. 
A fundamental result due to Gromov \cite{Gromov03} and Yomdin \cite{Yomdin87} establishes its equivalence to the \emph{algebraic entropy} \(h_{\mathrm{alg}}(f)\), defined as the logarithm of the spectral radius of the linear pullback operation \(f^*\) acting on the cohomology group \(\bigoplus_{k=0}^{n} H^{k,k}(M,\bC)\):
\[
h_{\mathrm{alg}}(f) \coloneqq \log \max_{0\leq k\leq n} \rho(f^*|_{H^{k,k}(M,\bC)}).
\]
For each \(k\), the quantity \(\rho(f^*|_{H^{k,k}(M,\bC)})\) is called the \emph{\(k\)-th dynamical degree} of \(f\) and denoted by \(\lambda_k(f)\); see \cref{defi:dynamical-degrees} for an equivalent definition for projective \(M\). 
We say that \(f\) is of \emph{positive entropy}, if \(h_{\mathrm{alg}}(f)>0\), or equivalently, if its first dynamical degree \(\lambda_1(f)\) is greater than $1$. 
Denote the full automorphism group of \(M\) by \(\Aut(M)\). 
A subgroup \(G\leq\Aut(M)\) is of \emph{positive entropy}, if all the elements of \(G\backslash \{\id\}\) are of positive entropy.

Dinh and Sibony proved in \cite{DS04} that every abelian subgroup \(G\leq \Aut(M)\) of positive entropy is free of rank \(r(G)\le n-1\). Subsequently, Zhang \cite{Zhang09} established a theorem of Tits type for \(\Aut(M)\) and extended \cite{DS04} to the solvable group case. Since then, algebraic dynamics on systems with maximal dynamical rank has been intensively studied (see, e.g., \cite{Din12,Zhang13,CWZ14,DHZ15,Zhang16,Hu20,Zho22,Zho23}). 

On the other hand, topological dynamics on higher rank abelian group actions has also been investigated for decades and various rigidity theorems hold on such systems (see, e.g., \cite{KS94,KK01,KKRH14}).
However, arithmetic dynamics on these higher rank abelian group actions does not seem to attract as much attention.
In this article, we aim to explore the system of maximal dynamical rank from this perspective. 
Specifically, we construct a canonical height associated with a $G\cong \bZ^{n-1}$-action of positive entropy on an $n$-dimensional smooth projective variety \(X\), extending the works in \cite{Silverman91,Kawaguchi08} from surfaces to higher dimensions.

Throughout the article, unless otherwise stated, we will work over the field $\Qbar$ of algebraic numbers.
Below is our main result; see \cref{thm:A-full} for its more precise form and \S\ref{sub-height} for the terminology involved.

\begin{theorem}
\label{thm:A}
Let \(X\) be a smooth projective variety of dimension \(n\ge 2\) over \(\Qbar\).
Let \(G\cong\bZ^{n-1}\) be a free abelian group of automorphisms of $X$ such that every nontrivial element of \(G\) has positive entropy.
Then there exist a function $\widehat{h}_G$ on $X(\Qbar)$ and a \(G\)-invariant Zariski closed proper subset $Z$ of $X$ such that:
\begin{enumerate}[label=\emph{(\arabic*)}, ref=(\arabic*)]
\item $\widehat{h}_G$ is a Weil height function corresponding to a nef and big $\bR$-divisor on $X$;
\item $\widehat{h}_G$ satisfies the Northcott property on $(X\backslash Z)(\Qbar)$;
\item For any $x\in (X\backslash Z)(\Qbar)$, one has $\widehat{h}_G(x) = 0$ if and only if $x$ is $g$-periodic for any $g\in G$.
\end{enumerate}
\end{theorem}

We refer to the function \(\widehat{h}_G\) in \cref{thm:A} as \emph{a canonical height function} associated with the abelian group \(G\) of maximal dynamical rank (cf.~\cref{thm:nef-can-ht}). 
Canonical height theory was initially developed by N\'eron and Tate for abelian varieties in the 1960s, and it has proved to be a crucial concept in arithmetic geometry. 
In his pioneering work \cite{Silverman91}, Silverman constructed canonical height functions on certain so-called Wehler K3 surfaces that are defined by the smooth complete intersection of two divisors of type \((1,1)\) and \((2,2)\) in \(\bP^2 \times \bP^2\).
Since then, there has been an extensive body of work on canonical height in arithmetic dynamics (see \cite{GTM241} and references therein).

Before delving into the applications of our \cref{thm:A}, let us provide a brief sketch of the strategy behind its proof. 

\begin{remark}
The idea of our proof of \cref{thm:A} draws significant inspiration from the works of Dinh--Sibony \cite{DS04} and Call--Silverman \cite{CS93}.
To be specific, we first follow the approach in \cite{DS04} to construct $n$ distinguished automorphisms $g_i$ in $G$ and $n$ common nef $\bR$-eigendivisors $D_i$ on $X$.
Each $g_i^*$ expands $D_i$ (up to $\bR$-linear equivalence) and shrinks $D_j$ with $j\neq i$ (up to numerical equivalence).
See \cref{thm:nef-big} for details. 

Furthermore, as established in \cite{CS93}, for each pair $(g_i,D_i)$ as described above, there exists a nef canonical height function $\widehat{h}_{D_i,g_i}$.
As the sum $\sum_i D_i$ forms a nef and big $\bR$-divisor, the corresponding sum $\sum_{i} \widehat{h}_{D_i,g_i}$ of nef canonical height functions turns out to be a Weil height function (denoted by \(\widehat{h}_G\)) satisfying the Northcott property.
See \cref{thm:A-full} for details.
\end{remark}

As the first application of our canonical height, we provide a positive answer to the Kawaguchi--Silverman conjecture under the assumption of maximal dynamical rank. In general, for an arithmetic dynamical system \((X,f)\) over \(\overline{\bQ}\) and a point \(x\in X(\overline{\bQ})\), we use an ample height function to measure the arithmetic complexity of the orbit \(\mathcal{O}_f(x)\).
The guiding principle is that "geometry governs arithmetic" in the sense that the height growth rate along the orbit is controlled by the first dynamical degree.
The conjecture predicts that if \(x\) is sufficiently complicated (e.g., has a Zariski dense orbit), then the height growth achieves the maximum. 
For precise definitions and the statement of the conjecture, we refer the reader to \S\ref{sub-KSC}; see also \cite{DGHLS22} for a higher-dimensional analog.

\begin{corollary}
\label{cor:KSC}
Under the assumption of \cref{thm:A}, for each \(g\in G\), the arithmetic degree \(\alpha_g(x)\) of $g$ at any \(x\in X(\Qbar)\) equals the first dynamical degree \(\lambda_1(g)\) of $g$, whenever the forward $g$-orbit $\cO_g(x)\coloneqq\{g^m(x)\colon m\in \bZ_{\geq 0}\}$ of \(x\) is Zariski dense in $X$.
\end{corollary}

\begin{remark}
The Kawaguchi--Silverman conjecture (i.e., \cref{conj:KSC}) has been successfully proved in many cases, leveraging canonical height theory. Notable contributions include works by \cite{Silverman91,CS93,Kawaguchi06,Kawaguchi08,KS14-exa,KS16-abelian,Shibata19,LS21}, among others.
Our approach to proving \cref{cor:KSC} closely follows this well-established strategy.
In comparison, an alternative geometric approach to address \cref{conj:KSC} involves the Equivariant Minimal Model Program, as developed by Meng and Zhang.
See their survey paper \cite{MZ23-survey} and the references therein. 
For the reader interested in this geometric perspective, we recommend exploring \cite{MSS18, Matsuzawa20-Adv, MZ22, MY22, MZ23-ksc}.
For a comprehensive overview of the current state of \cref{conj:KSC}, we refer to a recent survey paper \cite{Mat23-survey} and the references therein.
\end{remark}

Our second application involves an investigation of the bounded height property for periodic points, as well as an exploration of the height counting function for non-periodic points.

\begin{corollary}[{cf.~\cite[Theorem~D]{Kawaguchi08} and \cite[Proposition~3]{KS16}}]
\label{cor:arithmetic-estimate}
Under the assumption of \cref{thm:A} and with notation therein, the following assertions hold for any \(\id \neq g\in G\). 
\begin{enumerate}[label=\emph{(\arabic*)}, ref=(\arabic*)]
\item \label{cor:arithmetic-estimate1} The subset below
\[
\{x\in (X\backslash Z)(\Qbar) : x~\textup{is}~g\textup{-periodic}\}
\]
is of bounded height (see \cref{defn-bounded-height}).
\item \label{cor:arithmetic-estimate2} For any \(x\in (X\backslash Z)(\Qbar)\) which is not \(g\)-periodic, one has
\[
\lim_{T\to+\infty}\frac{\#\{ m\in \bZ_{\geq 0} : h_{H_X}(g^m(x)) \leq T \}}{\log T} = \frac{1}{\log\lambda_1(g)}.
\]
\end{enumerate}
\end{corollary}

This article is organized as follows.
In \Cref{section:prelim}, we provide a review of essential concepts and results, including weak numerical equivalence, stable base loci, Weil's height theory, and dynamical/arithmetic degrees.
In \Cref{section:construction}, we proceed to construct $n$ distinguished automorphisms $g_i$ in $G$ and a nef and big \(\bR\)-divisor as a sum of common nef eigendivisors $D_i$, possessing favorable properties from an arithmetic perspective (see \cref{thm:nef-big,rmk_DS_eigendivisor}).
Following this, in \Cref{section:can-ht}, we define a canonical height function associated with the abelian group \(G\) of maximal dynamical rank.
Finally, we present the proofs of \cref{thm:A,cor:KSC,cor:arithmetic-estimate}.

We conclude the introduction with the following remark.

\begin{remark}[About the generalization of \cref{thm:A}]
It is noteworthy that extending our main result to normal projective varieties over \(\Qbar\) poses no essential difficulties, using intersection theory and Weil's height theory on \(\bR\)-Cartier divisors.
Furthermore, the extension of our main result to global fields in positive characteristic is also possible using the language in \cite{Tru20,Dang20,Hu20}, along with the result from \cite{Hu-lc}.
However, given that the primary focus of this article is to present a distinctive perspective for studying abelian group actions of maximal dynamical rank, we choose to concentrate on smooth projective varieties over \(\Qbar\).
This decision is made in the interest of maintaining clarity and simplicity in our exposition.
We anticipate that this alternative viewpoint may provide valuable insights into the classification problem of abelian group actions with lower dynamical rank (see \cite[Problem~1.5]{Din12}).
\end{remark}

\section*{Acknowledgements}

The authors would like to express their gratitude to the organizers of the conference "Recent Developments in Algebraic Geometry, Arithmetic and Dynamics Part 2" held in Singapore in August 2023, for providing a platform for fruitful discussions.
The first author extends special thanks to the Fields Institute for its warm hospitality during a visit in Fall 2023.

\section{Preliminaries}
\label{section:prelim}

We start with notation and terminology. 
Let \(X\) be a smooth projective variety of dimension $n$ over \(\Qbar\).
The symbols \(\sim\), \(\sim_\bQ\), \(\sim_\bR\), \(\approx\), and \(\equiv\) stand for $\bZ$-, $\bQ$-, $\bR$-linear equivalence, algebraic equivalence, and numerical equivalence, respectively.

For any $0\leq i\leq n$, denote by $\CH^i(X)$ the Chow group of algebraic cycles of codimension~$i$ on $X$ modulo linear equivalence.
It is well-known in intersection theory that $\CH(X) \coloneqq \bigoplus_{i=0}^n \CH^i(X)$ is a graded commutative ring with respect to the intersection product, called the Chow ring.
When working with a coefficient field $\bK = \bQ$ or $\bR$, we write
\[
\CH^i(X)_\bK \coloneqq \CH^i(X) \otimes_\bZ \bK.
\]
In particular, when $i=1$, the Chow group $\CH^1(X)$ coincides with the Picard group $\Pic(X)$ of $X$. 
Denote by $\Pic^0(X)$ the subgroup of $\Pic(X)$ consisting of all integral divisors on $X$ algebraically equivalent to zero (modulo linear equivalence);
it has a structure of an abelian variety.
The quotient group
\[
\NS(X)\coloneqq\Pic(X)/\Pic^0(X)
\]
is called the \emph{N\'eron--Severi group} of \(X\), which is a finitely generated abelian group.

Let \(\pic^0_{X/\Qbar}\) denote the Picard variety of $X$ over \(\Qbar\), i.e., the neutral connected component of the Picard group scheme \(\pic_{X/\Qbar}\) of $X$ over \(\Qbar\);
it is also the dual abelian variety of the Albanese variety \(\Alb(X)\) of $X$ (see \cite[Theorem~9.5.4 and Remark~9.5.25]{Kle05}).
Note that the group of the $\Qbar$-points of the Picard variety \(\pic^0_{X/\Qbar}\) is exactly $\Pic^0(X)$.

For any $0\le i\le n$, denote by $\N^i(X)$ the finitely generated free abelian group of algebraic cycles of codimension~$i$ on $X$ modulo numerical equivalence, i.e.,
\[
\N^i(X) \coloneqq \CH^i(X)/\!\num.
\]
For $\bK=\bQ$ or $\bR$, denote by $\N^i(X)_\bK \coloneqq \N^i(X) \otimes_\bZ \bK$ the associated finite-dimensional $\bK$-vector space.
It is also well-known that when $i=1$,
\[
\N^1(X)_\bK = \NS(X)_\bK \coloneqq \NS(X) \otimes_\bZ \bK.
\]
The \emph{Picard number} $\rho(X)$ of $X$ is defined as the rank of $\N^1(X)$ or $\dim_\bR \N^1(X)_\bR$.
We endow $\N^1(X)_\bR$ with the standard Euclidean topology and we fix a norm $\norm{\cdot}$ on it.

A divisor $D$ on $X$ is \emph{nef}, if the intersection number $D\cdot C$ is nonnegative for any curve $C$ on $X$.
The cone of all nef $\bR$-divisors in $\N^1(X)_\bR$ is called the \emph{nef cone} $\Nef(X)$ of $X$, which is a salient closed convex cone of full dimension.
Its interior is called the \emph{ample cone} $\Amp(X)$ of $X$.
An integral divisor $D$ on $X$ is \emph{big}, if the linear system $|mD|$ of some multiple of \(D\) induces a birational map $\Phi_{|mD|}$ from $X$ onto its image.
An $\bR$-divisor $D$ is \emph{big}, if it is a positive combination of integral big divisors.
It is well-known that for any nef $\bR$-divisor $D$ on $X$, it is big if and only if $D^n>0$.

\subsection{Weak numerical equivalence}

It turns out to be convenient to consider the following notion (implicitly) used in Dinh--Sibony \cite{DS04} and Zhang \cite{Zhang09}.

\begin{definition}[Weak numerical equivalence]
\label{def:weak-numerical-equiv}
Let \(X\) be a smooth projective variety of dimension $n$ over \(\Qbar\).
An algebraic cycle $Z$ of codimension~$i$ on $X$ is called {\it weakly numerically trivial} and denoted by $Z\wnum 0$, if
\[
Z\cdot H_1\cdots H_{n-i} = 0
\]
for all ample (and hence for all) \(\bR\)-divisors $H_1,\ldots,H_{n-i}$ on $X$.
\end{definition}

Clearly, weak numerical equivalence is coarser than numerical equivalence.
An important property of weak numerical equivalence, crucial for deducing the nonvanishing of intersection numbers of divisors, is the following result due to Dinh and Sibony.

\begin{lemma}[{cf.~\cite[Lemme~4.4]{DS04}}]
\label{lemma:DS-4.4}
Let $D_1,\dots,D_j,D'_j$ be nef $\bR$-divisors on $X$ with $1\leq j\leq n-1$ such that $D_1 \cdots D_{j-1}\cdot D_j \not\wnum 0$ and $D_1 \cdots D_{j-1}\cdot D'_j \not\wnum 0$.
Let $f$ be an automorphism of $X$ such that
\begin{align*}
f^*(D_1 \cdots D_{j-1}\cdot D_j) & \wnum \lambda D_1 \cdots D_{j-1}\cdot D_j \ \text{ and} \\
f^*(D_1 \cdots D_{j-1}\cdot D'_j) & \wnum \lambda' D_1 \cdots D_{j-1}\cdot D'_j
\end{align*}
with positive real numbers $\lambda \neq \lambda'$.
Then $D_1 \cdots D_{j-1}\cdot D_j\cdot D'_j \not\wnum 0$.
\end{lemma}

\subsection{Stable and augmented base loci}

In the course of our construction of a canonical height function associated with an abelian group $G$ of automorphisms of $X$, we actually first construct a nef and big $\bR$-divisor on $X$.
In dealing with the height functions of $\bR$-divisors, it becomes necessary to consider the so-called augmented base loci of \(\bR\)-divisors.

\begin{definition}[Augmented base loci]\label{defn-augmented-bl}
Let \(X\) be a smooth projective variety over $\Qbar$.
The \emph{stable base locus} $\bB(D)$ of a \(\bQ\)-divisor \(D\) is the Zariski closed subset of \(X\) defined by
\[
\bB(D)\coloneqq \bigcap_{m\geq 1, \, mD \text{~is Cartier}}\Bs(mD),
\]
where \(\Bs(mD)\) denotes the base locus of the linear system \(|mD|\).
It is an elementary fact that there is an $M\geq 1$ such that $\bB(D) = \Bs(MD)$.

The \emph{augmented base locus} $\bB_+(D)$ of an \(\bR\)-divisor \(D\) is the Zariski closed subset of \(X\) defined by
\[
\bB_+(D)\coloneqq \bigcap_{A} \bB(D-A),
\]
where the intersection is taken over all ample $\bR$-divisors \(A\) such that $D-A$ is a $\bQ$-divisor.
\end{definition}

For a detailed study of augmented base loci, we direct the reader to \cite{ELMNP06} and references therein.
Here we only state a few of them which will be utilized in the proof of \cref{thm:A} and its corollaries.

\begin{proposition}[{cf.~\cite[Propositions~1.4 and 1.5, Example~1.7]{ELMNP06}}]
\label{prop:B_+}
Let \(X\) be a smooth projective variety over \(\Qbar\). 
Then the following assertions hold.
\begin{enumerate}[label=\emph{(\arabic*)}, ref=(\arabic*)]
\item \label{prop:B_+-big-iff-nx} For any \(\bR\)-divisor \(D\) on \(X\), it is big if and only if \(\bB_+(D)\neq X\). 
\item \label{prop:B_+-num} If $D_1$ and $D_2$ are numerically equivalent $\bR$-divisors on $X$, then $\bB_+(D_1) = \bB_+(D_2)$.
\item \label{prop:B_+-modification}For any $\bR$-divisor $D$ on $X$, there is a positive number $\varepsilon$ such that for any ample $\bR$-divisor $A$ with $\norm{A}\leq \varepsilon$ and such that $D-A$ is a $\bQ$-divisor, \(\bB_+(D) = \bB(D-A)\).
\end{enumerate}
\end{proposition}

Lesieutre and Satriano \cite{LS21} observed that for two nef $\bR$-divisors $D_1,D_2$ on $X$, the augumented base locus of $\bB_+(a_1D_1 + a_2D_2)$ is independent of the positive coefficients $a_1$ and $a_2$.
By induction, we can easily deduce the following.

\begin{lemma}[{cf.~\cite[Lemma~2.16]{LS21}}]
\label{lemma:LS-2.16}
Let $D_1,\dots,D_m$ be nef $\bR$-divisors on $X$.
Then for any $a_1,\dots,a_m>0$, one has
\[
\bB_+(a_1D_1 + \dots + a_mD_m) = \bB_+(D_1 + \dots + D_m).
\]
\end{lemma}

\subsection{Weil height and canonical height}
\label{sub-height}

We refer to \cite[Part B]{GTM201} for an introduction to Weil's height theory.
Among others, we collect some important facts from there.

\begin{theorem}[{cf.~\cite[Theorems~B.3.2, B.3.6, and B.5.9]{GTM201} and \cite[Theorem~1.1.1]{Kawaguchi06-Crelle}}]
\label{thm:weil-height}
Let \(X\) be a smooth projective variety over \(\Qbar\).
Then there exists a unique homomorphism 
\[
h_X\colon \Pic(X)_\bR \to {\raisebox{.1em}{$\{\emph{functions } X(\Qbar) \to \bR\}$}\big/\raisebox{-.1em}{$\{\emph{bounded functions } X(\Qbar) \to \bR\}$}}
\]
satisfying the following properties.
\begin{enumerate}[label=\emph{(\roman*)}, ref=(\roman*)]
\item \emph{(Normalization)} Let \(D\) be a very ample divisor on $X$ and \(\phi_D\colon X\injmap \bP^N\) the associated embedding.
Then we have
\[
h_{X,D} = h\circ \phi_D + O(1),
\]
where \(h\) is the absolute logarithmic height on \(\bP^N\) (see \cite[Definition, Page~176]{GTM201}). 

\item \label{thm:weil-height-functoriality} \emph{(Functoriality)} Let \(\pi\colon X\to Y\) be a morphism of smooth projective varieties and \(D_Y\in \Pic(Y)_\bR\).
Then we have
\[
h_{X,\pi^*\!{D_Y}} = h_{Y,D_Y}\circ \pi + O(1).
\]

\item \label{thm:weil-height-additivity} \emph{(Additivity)} Let \(D_1,D_2\) be $\bR$-divisors on $X$.
Then we have
\[
h_{X,D_1+D_2} = h_{X,D_1} + h_{X,D_2} + O(1).
\]

\item \label{thm:weil-height-positivity} \emph{(Positivity)} Let \(D\) be an effective integral divisor on $X$.
Then \(h_{X,D} \geq O(1)\) outside the base locus \(\Bs(D)\) of \(D\).

\item \label{thm:weil-height-alg-equiv} \emph{(Algebraic equivalence)} Let $H,D\in \Pic(X)_\bR$ be $\bR$-divisors with $H$ ample and $D$ algebraically equivalent to zero.
Then there is a constant $C>0$ such that
\[
h_{X,D} \leq C \sqrt{h_{X,H}^+},
\]
where $h_{X,H}^+ \coloneqq \max(1,h_{X,H})$.
\end{enumerate}
\end{theorem}

\begin{remark}
(1) It is worth mentioning that the $O(1)$ terms only depend on varieties, divisors, and morphisms, but are independent of rational points of varieties.
This is why we omit the points $x\in X(\Qbar)$ in various height equations.
See \cite[Remarks~B.3.2.1(ii)]{GTM201}.

(2) When the ambient variety $X$ is clear, we often use $h_D$ to stand for $h_{X,D}$ for simplicity.
\end{remark}

The following finiteness property, originally established by Northcott for integral ample divisors, becomes a fundamental tool in Weil's height theory (see \cite[Theorem~B.3.2(g)]{GTM201} and \cite[Theorem~1.1.2]{Kawaguchi06-Crelle}).
Lesieutre and Satriano proved a version for big $\bR$-divisors (see \cite[Lemma~2.26]{LS21}).
For the reader's convenience, we restate it here.

\begin{theorem}[Northcott finiteness property]
\label{thm:Northcott}
Let \(X\) be a smooth projective variety over \(\Qbar\) and \(D\) a big \(\bR\)-divisor on \(X\).
Then for any \(d\in \bZ_{>0}\) and \(N\in \bR\), the set
\[
\big\{ x\in (X\backslash\bB_+(D))(\Qbar) : [\bQ(x):\bQ] \leq d, \, h_D(x)\leq N \big\}
\]
is finite, where $\bB_+(D)$ denotes the augmented base locus of $D$.
In particular, if \(D\) is an ample \(\bR\)-divisor, then \(\bB_+(D)=\emptyset\) and this is the usual Northcott finiteness property.
\end{theorem}

\begin{definition}[Bounded height]
\label{defn-bounded-height}
Fix an ample divisor \(H_X\) on a smooth projective variety \(X\) over \(\Qbar\). 
A subset \(S\subseteq X(\Qbar)\) is called a set of \emph{bounded height}, if there is a constant \(C\) such that \(h_{H_X}(s)\leq C\) for all \(s\in S\). 
This property is independent of the choice of the ample divisor \(H_X\).
\end{definition}

The Northcott finiteness property \cref{thm:Northcott} implies that if \(S\subseteq X(\Qbar)\) is a set of bounded height, then \(\{s\in S : [\bQ(s) : \bQ]\leq d\}\) is finite for every positive integer \(d\).

In what follows, we recall a classical result on the canonical height associated with a single endomorphism due to Call and Silverman \cite[Theorem~1.1]{CS93}, though in the book \cite{GTM201} the authors still attributed it to N\'eron and Tate.

\begin{theorem}[{cf.~\cite[Theorem~B.4.1]{GTM201}}]
\label{thm:nef-can-ht}
Let $f\colon X \to X$ be an endomorphism of a smooth projective variety $X$ over $\Qbar$. 
Let $D$ be an $\bR$-divisor on $X$ such that $f^*D \sim_\bR \alpha D$ for some $\alpha>1$.
Then there is a unique function $\widehat{h}_{D,f}\colon X(\Qbar)\to\bR$, called \emph{the canonical height} on $X$ with respect to $f$ and $D$, defined by
\[
\widehat{h}_{D,f}(x) \coloneqq \lim_{m\to \infty} \frac{h_D(f^m(x))}{\alpha^m},
\]
satisfying the following properties:
\begin{enumerate}[label=\emph{(\arabic*)}, ref=(\arabic*)]
\item $\widehat{h}_{D,f} = h_{D} + O(1)$;
\item $\widehat{h}_{D,f}\circ f = \alpha \, \widehat{h}_{D,f}$.
\end{enumerate}
\end{theorem}

\begin{remark}
In the above \cref{thm:nef-can-ht}, if we replace $\bR$-linear equivalence with algebraic or numerical equivalence --- often more pratical to verify --- Kawaguchi and Silverman introduced a "canonical height" in a similar manner.
Specifically, they showed that given $f^*D' \num \beta D'$ for some $\bR$-divisor $D'$ and some $\beta > \sqrt{\lambda_1(f)}$, the limit
\[
\widehat{h}_{D',f}(x) \coloneqq \lim_{m\to \infty} \frac{h_{D'}(f^m(x))}{\beta^m}
\]
exists and satisfies the following properties:
\begin{enumerate}[label=(\arabic*), ref=(\arabic*)]
\item $\widehat{h}_{D',f} = h_{D'} + O\big(\sqrt{h_{H}^+}\big)$;
\item $\widehat{h}_{D',f}\circ f = \beta \, \widehat{h}_{D',f}$.
\end{enumerate}
See \cite[Theorem~5]{KS16}.
Its proof relies on their height estimate \cite[Theorem~24]{KS16}.
Yet, the correct proof of this height estimate is due to Matsuzawa (see \cite[Theorem~1.4]{Matsuzawa20}).
Further, we remark that if one needs to control the $O(\sqrt{h_H^+})$ term, this height estimate is inevitable (see, e.g., \cref{lemma:linear-equiv,lemma:nef-can-ht}).
\end{remark}

\subsection{Dynamical degrees and arithmetic degrees}\label{sub-KSC}

In this subsection, we let \(f\colon X\to X\) be a surjective endomorphism of a smooth projective variety \(X\) of dimension $n$ over \(\Qbar\) and \(H_X\) an ample divisor on $X$.
We recall the definitions of dynamical degrees and arithmetic degrees of this algebraic dynamical system \((X,f)\), as well as the Kawaguchi--Silverman conjecture which reveals the relationship between these two dynamical invariants.

\begin{definition}[Dynamical degrees]
\label{defi:dynamical-degrees}
For each \(0\leq k\leq n\), the \emph{\(k\)-th dynamical degree} of \(f\) is defined by
\begin{align*}
\lambda_k(f) \coloneqq \lim_{m\to\infty}((f^m)^*H_X^k \cdot H_X^{n-k})^{1/m} \in \bR_{\geq 1}.
\end{align*}
It is known that dynamical degrees are birational invariants and are independent of the choice of \(H_X\) (see \cite{DS05,Tru20,Dang20}).
They also satisfy the log-concavity property as follows:
\[
\lambda_{k}(f)^2\geq \lambda_{k-1}(f)\lambda_{k+1}(f) \text{ for } 1\leq k\leq n-1.
\]
The \emph{algebraic entropy} of \(f\) is defined by
\[
h_{\mathrm{alg}}(f)\coloneqq\max_{0\leq k\leq n}\log\lambda_k(f).
\]
We refer to \cite{Din12} for more details on dynamical degrees.
\end{definition}

\begin{definition}[Arithmetic degrees]
\label{defn-arith-deg}
Let \(h_{H_X}\) be an absolute logarithmic Weil height function associated with \(H_X\).
Set \(h_{H_X}^+ = \max(1,h_{H_X})\).
For each \(x\in X(\Qbar)\), we define the \emph{arithmetic degree} of $f$ at $x$ by
\[
\alpha_f(x)\coloneqq \lim_{m\to\infty} h_{H_X}^+(f^m(x))^{1/m} \in \bR_{\geq 1}.
\]
It is known that the limit always exists and is also independent of the choice of the ample divisor \(H_X\) (see \cite[Theorem~3]{KS16-abelian} and \cite[Proposition~12]{KS16}, respectively).
\end{definition}

The following conjecture due to Kawaguchi and Silverman asserts the properties of arithmetic degrees.
Throughout this paper, we shall only consider this conjecture for automorphisms.
We refer to \cite[Conjecture~6]{KS16} for a general version of dominant rational self-maps.

\begin{conjecture}[{cf.~\cite[Conjecture~6]{KS16}}]
\label{conj:KSC}
Let \(f\colon X\to X\) be a surjective endomorphism of a smooth projective variety \(X\) over \(\Qbar\).
Then for any point \(x\in X(\Qbar)\), if the forward $f$-orbit \(\cO_f(x)\coloneqq\{f^m(x)\colon m\in \bZ_{\geq 0}\}\) of $x$ is Zariski dense in \(X\), then 
\[
\alpha_f(x)=\lambda_1(f).
\]
\end{conjecture}

\section{Construction of distinguished automorphisms and divisors}
\label{section:construction}

The whole section is dedicated to the construction of $n$ distinguished automorphisms in the rank $n-1$ free abelian subgroup $G\leq \Aut(X)$ of positive entropy and $n$ nef \(\bR\)-divisors whose sum is a big \(\bR\)-divisor (see \cref{thm:nef-big}).
Consequently, we can define a canonical height in \Cref{section:can-ht}.

\subsection{Commuting families of linear maps preserving cones}
\label{subsec:commuting-family}

Throughout this subsection, $V$ is a finite-dimensional $\bR$-vector space and $C$ is a \emph{salient} closed convex cone of $V$ of \emph{full dimension} (i.e., $C\cap (-C)=\{0\}$ and $C$ spans $V$).
We recall a few facts on them from linear algebra which are crucial to the construction in \S\ref{subsec:construction}.
The following is Birkhoff's generalization of the classic Perron--Frobenius theorem.

\begin{theorem}[cf.~\cite{Birkhoff67}]
\label{thm:Birkhoff}
Let $\varphi \in \End(V)$ be an $\bR$-linear endomorphism of $V$ such that $C$ is $\varphi$-invariant, i.e., $\varphi(C) \subseteq C$.
Then the spectral radius $\rho(\varphi)$ is an eigenvalue of $\varphi$ and there is an eigenvector $v_\varphi \in C$ of $\varphi$ associated with $\rho(\varphi)$.
\end{theorem}

It is well-known that a commuting family $\cF \subseteq \Mat_n(\bC)$ of complex matrices possesses a nonzero common eigenvector $v \in \bC^n$.
Below is its analog when $\cF$ is a family of real matrices preserving a salient closed convex cone of full dimension.
It is essentially due to Dinh and Sibony (see \cite[Proposition~4.1]{DS04}).

\begin{proposition}
\label{prop:common-ev}
Let $\cF \subseteq \End(V)$ be a commuting family of $\bR$-linear endomorphisms of $V$ such that $C$ is $\cF$-invariant (i.e., $\psi(C) \subseteq C$ for any $\psi \in \cF$).
Then for any $\varphi \in \cF$, there exists a nonzero vector $v_\varphi \in C$ such that
\begin{enumerate}[label=\emph{(\arabic*)}, ref=(\arabic*)]
\item \label{prop:common-ev-1} for every $\psi \in \cF$, $\psi(v_\varphi) \in \bR_{\geq 0} \cdot v_\varphi$, i.e., $v_\varphi$ is a common eigenvector for all $\psi \in \cF$ associated with some nonnegative eigenvalues; and moreover,
\item \label{prop:common-ev-2} $\varphi(v_\varphi) = \rho(\varphi)v_\varphi$.
\end{enumerate}
\end{proposition}
\begin{proof}
Let $\varphi \in \cF$ be fixed.
By Birkhoff's \cref{thm:Birkhoff}, the spectral radius $\rho(\varphi)$ of $\varphi$ is an eigenvalue of $\varphi$ and the corresponding eigenvector can be chosen to lie in $C$.
In particular,
\[
C_\varphi \coloneqq \{v \in C : \varphi(v) = \rho(\varphi) v\}
\]
is a nonzero salient closed convex cone of $V$.
It is easy to see that $C_\varphi$ is $\cF$-invariant.
Indeed, for any $\psi \in \cF$ and any $v \in C_\varphi \subseteq C$, by assumption that $C$ is $\cF$-invariant, we have $\psi(v) \in C$.
Then it follows from the commutativity of $\cF$ that
\[
\varphi(\psi(v)) = \psi(\varphi(v)) = \psi(\rho(\varphi)v) = \rho(\varphi) \psi(v).
\]
Hence $\psi(v) \in C_\varphi$ by the definition of $C_\varphi$.

Denote by $V_\varphi$ the $\bR$-vector subspace of $V$ spanned by $C_\varphi$.
Clearly, $V_\varphi$ is nonzero, contained in the eigenspace of $\varphi$ associated with the eigenvalue $\rho(\varphi)$, and $\cF$-invariant, since so is $C_\varphi$.
It suffices to show that there exists a common eigenvector $v_\varphi \in C_\varphi$ for every $\psi \in \cF$.
In other words, let $\psi$ be arbitrary in $\cF$ and denote
\[
\widetilde{C}_\psi \coloneqq \{v \in C_\varphi : \psi(v) = \chi_\psi v, \, \chi_\psi \in \bR_{\geq 0}\},
\]
which is a nonzero salient closed (possibly nonconvex) cone of $V_\varphi$.
Then it remains to show that
\begin{equation}
\label{eq:infinite-intersection}
\bigcap_{\psi \in \cF} \widetilde{C}_\psi \neq \{0\},
\end{equation}
or equivalently, in the quotient space $\bP^+(V_\varphi)\coloneqq (V_{\varphi}\backslash\{0\})/\bR_{>0}$ (think of \(\bR_{>0}\) as the multiplicative subgroup of $\bR^*\coloneqq\bR\backslash\{0\}$),
\begin{equation}
\label{eq:infinite-intersection-Proj}
\bigcap_{\psi \in \cF} \bP^+(\widetilde{C}_\psi) \neq \emptyset,
\end{equation}
where each $\bP^+(\widetilde{C}_\psi)$ denotes the image of $\widetilde{C}_\psi\backslash\{0\}$ under the natural quotient map
\[
\pi\colon V_\varphi\backslash\{0\} \to \bP^+(V_\varphi).
\]
Note that $\bP^+(V_\varphi)$ endowed with the quotient topology is homeomorphic with the $(\dim V_\varphi-1)$-sphere and hence compact.
Moreover, \(\bP^+(\widetilde{C}_\psi)\) is closed in $\bP^+(V_\varphi)$ since so is \(\widetilde{C}_{\psi}\backslash \{0\}\) in $V_\varphi\backslash\{0\}$.
We are thus reduced to show \cref{eq:infinite-intersection-Proj} or \cref{eq:infinite-intersection} for any finite $\cF'\subseteq \cF$.

Suppose now that $\cF' = \{\psi_1,\ldots,\psi_m\}$ is finite.
We shall inductively construct pairs $(V_k, C_k)$, $1 \leq k \leq m$, satisfying the following properties:
\begin{enumerate}[label=(\roman*), ref=(\roman*)]
\item \label{decreasing_V} $V_0 \coloneqq V_\varphi \supseteq V_1 \supseteq \dots \supseteq V_m \neq \{0\}$ is a decreasing sequence of nonzero $\cF$-invariant $\bR$-vector subspaces of $V_\varphi$;
\item \label{decreasing_C} $C_0 \coloneqq C_\varphi \supseteq C_1 \supseteq \dots \supseteq C_m \neq \{0\}$ is a decreasing sequence of nonzero $\cF$-invariant salient closed convex cones in $V_\varphi$;
\item \label{nonzero_intersection} for each $1 \leq k \leq m$, $C_k$ spans $V_k$ and
\[
C_k \subseteq \bigcap_{i=1}^k \widetilde{C}_{\psi_i}.
\]
\end{enumerate}
As an immediate consequence of this construction, one gets $\cap_{i=1}^m \widetilde{C}_{\psi_i} \neq \{0\}$ since it contains a nonzero cone $C_m$, which completes the proof of \cref{prop:common-ev}.

We first construct $(V_1, C_1)$.
Recall that $V_0=V_\varphi$ and the cone $C_0=C_\varphi$ are both $\cF$-invariant and hence $\psi_1$-invariant.
Applying \cref{thm:Birkhoff} again to the triplet $(V_0, \psi_1|_{V_0}, C_0)$ yields that the spectral radius $\rho(\psi_1|_{V_0})$ of the restriction $\psi_1|_{V_0}$ is an eigenvalue of $\psi_1|_{V_0}$ and
\[
C_{1} \coloneqq \{v \in C_0 : \psi_1(v) = \rho(\psi_1|_{V_0}) v\}
\]
is a nonzero salient closed convex cone of $V_0$.
Again by the commutativity of $\cF$ and the $\cF$-invariance of $C_0$, we see that $C_1$ is $\cF$-invariant.
Denote by $V_1$ the $\bR$-vector subspace of $V_0$ spanned by $C_1$.
Then $V_1$ is also $\cF$-invariant and $C_1 \subseteq \widetilde{C}_{\psi_1}$ by definition.

By inductive hypothesis, suppose that we have constructed pairs $(V_i, C_i)$ for all $1 \leq i \leq k-1$ with $2 \leq k \leq m$.
We then construct $(V_k, C_k)$ satisfying all properties~\ref{decreasing_V} to \ref{nonzero_intersection}.
Note that $V_{k-1}$ and the spanning cone $C_{k-1}$ are both $\cF$-invariant and hence $\psi_k$-invariant.
It follows from \cref{thm:Birkhoff}, applied to the triplet $(V_{k-1}, \psi_k|_{V_{k-1}}, C_{k-1})$, that
\[
C_k \coloneqq \{v \in C_{k-1} : \psi_k(v) = \rho(\psi_k|_{V_{k-1}}) v\}
\]
is a nonzero salient closed convex cone of $V_{k-1}$.
Let $V_k$ denote the $\bR$-vector subspace of $V_{k-1}$ spanned by $C_k$.
Both $V_k$ and $C_k$ are $\cF$-invariant in a similar vein.
Moreover, as $C_k \subseteq C_{k-1} \cap \widetilde{C}_{\psi_k}$ by construction, the property~\ref{nonzero_intersection} for $C_k$ follows.
\end{proof}

\begin{remark}
\label{rmk:characters}
Note that in the above \cref{prop:common-ev}, if we replace $\cF \subseteq \End(V)$ with an abelian subgroup $\cG \leq \GL(V)$ of \(\bR\)-linear automorphism group of \(V\), then for any $\varphi \in \cG$, there is a common eigenvector $v_\varphi \in C$ such that for every $\psi \in \cG$, $\psi(v_\varphi) \in \bR_{>0} \cdot v_\varphi$ and $\varphi(v_\varphi) = \rho(\varphi)v_\varphi$.
It gives rise to a multiplicative group character \(\chi_\varphi \colon \cG \to (\bR_{>0}, \times)\) defined by $\psi(v_\varphi) = \chi_\varphi(\psi)v_\varphi$ for any $\psi\in \cG$; the character $\chi_\varphi$ also satisfies that \(\chi_\varphi(\varphi)=\rho(\varphi)\).
\end{remark}

A priori, it is still unknown that for different $\varphi$ and $\varphi'$, the above characters $\chi_{\varphi}$ and $\chi_{\varphi'}$ are distinct, neither the uniqueness of $\chi_\varphi$ for each $\varphi$.
Nonetheless, the following lemma shows that there are at most $\dim_\bR V$ distinct characters of $\cG$ constructed via common eigenvectors in $C$, whose proof is straightforward and omitted.

\begin{lemma}
\label{lemma:distinct-characters}
Let $\cG$ be an abelian subgroup of $\GL(V)$ such that $C$ is $\cG$-invariant.
As in \cref{rmk:characters}, let \(\chi_i \colon \cG \to (\bR_{>0}, \times)\) with \(1\leq i\leq m\) be group characters of $\cG$ associated with some common eigenvectors \(v_i \in C\).
Then the following statements are equivalent:
\begin{enumerate}[label=\emph{(\arabic*)}, ref=(\arabic*)]
\item \label{lemma:distinct-characters-1} \(\chi_1,\dots,\chi_m\) are mutually distinct group characters;
\item \label{lemma:distinct-characters-2} \(v_1,\dots,v_m\) are mutually noncollinear;
\item \label{lemma:distinct-characters-3} \(v_1,\dots,v_m\) are linearly independent.
\end{enumerate}
\end{lemma}

\begin{lemma}
\label{lemma:chi-achieves-rho}
Let $\{\chi_1,\dots,\chi_m\}$ denote the set of all distinct multiplicative group characters of $\cG \leq \GL(V)$ with $m\leq \dim_\bR V$, where each $\chi_i$ is associated with a common eigenvector $v_i\in C$. 
Then for any \(\varphi \in \cG\), there exists some \(1 \leq i \leq m\) such that \(\chi_i(\varphi) = \rho(\varphi)\), i.e., \(\varphi(v_i) = \rho(\varphi)v_i\).
\end{lemma}
\begin{proof}
Suppose to the contrary that there exists some \(\varphi \in \cG\) such that \(\chi_i(\varphi) \neq \rho(\varphi)\) for any \(1\leq i\leq m\).
Then by \cref{prop:common-ev}, there is a group character \(\chi_\varphi\) of $\cG$ associated with a common eigenvector \(v_\varphi \in C\) such that \(\chi_\varphi(\varphi) = \rho(\varphi)\), i.e., \(\varphi(v_\varphi) = \rho(\varphi)v_\varphi\).
Clearly, this new character \(\chi_\varphi\) is different from any other \(\chi_i\), a contradiction.
\end{proof}

We will use the following lemma to strengthen eigenequations modulo numerical equivalence to equations modulo $\bR$-linear equivalence (see \cref{lemma:linear-equiv}).
For the sake of completeness, we provide its proof here. 

\begin{lemma}
\label{lem-operator-infinite-dim}
Let \(R\subseteq S\) be integral domains, \(M\) an (unnecessarily finitely generated) \(R\)-module and \(\varphi\colon M\to M\) an \(R\)-linear map such that \(P(\varphi)=0\) for some polynomial \(P(t)\in R[t]\).
Denote the field of fractions of \(S\) by \(K\). 
Let \(M_K\coloneqq M\otimes_R K\) be the vector space over \(K\) and \(\varphi_K\coloneqq \varphi\otimes_R\id_K\) the induced $K$-linear map on \(M_K\).
Let \(s\in S\) such that \(P(s)\neq 0\) in $S$.
Then \(\varphi_K-s\id_{M_K}\) is an isomorphism of $M_K$.
\end{lemma}
\begin{proof}
We first prove that \(\varphi_K-s\id_{M_K}\) is injective.
Let $v\in M_K$ such that $(\varphi_K-s\id_{M_K})(v)=0$, i.e., $\varphi_K(v) = s v$.
It is easy to verify that \(P(\varphi_K)=(P(\varphi))_K=0\).
Then we have 
\[
0 = P(\varphi_{K})(v) = P(s) v.
\]
As $P(s)\neq 0$ is invertible in $K$, it follows that $v=0$.

We next show that \(\varphi_K-s\id_{M_K}\) is also surjective.
Denote the degree of the polynomial \(P(t)\) by $n$. 
Let \(w\in M_K\) be arbitrary.
As \(P(\varphi_K)(w)=0\), there exist \(c_0,\dots,c_{n-1}\in K\) (or rather, in the field of fractions of $R$), such that
\[
\varphi_K^n(w) = c_{n-1}\varphi_K^{n-1}(w) + \cdots + c_1\varphi_K(w) + c_0 w.
\]
In other words, the vector space $B$ generated by $\{w,\varphi_K(w),\varphi_K^2(w),\dots\}$ over $K$ is a finite-dimensional $\varphi_K$-invariant subspace of $M_K$.
It follows that the restriction map \((\varphi_K-s\id_{M_K})|_B\) of \(\varphi_K-s\id_{M_K}\) on $B$ is also injective.
As $B$ is of finite dimension, \((\varphi_K-s\id_{M_K})|_B\) is surjective.
Since $w\in M_K$ is arbitrary, we see that \(\varphi_K-s\id_{M_K}\) itself is surjective.
\end{proof}

\subsection{Automorphisms and divisors associated with abelian groups of maximal dynamical rank}
\label{subsec:construction}

Below is the main result of this section.
Given an abelian subgroup $\bZ^{n-1}\cong G\leq \Aut(X)$ of positive entropy, we construct $n$ distinguished automorphisms in $G$ and a nef and big $\bR$-divisor associated with $G$.
This construction forms a crucial ingredient for defining a canonical height in \Cref{section:can-ht} (see \cref{thm:A-full}).

\begin{theorem}
\label{thm:nef-big}
Let $X$ be a smooth projective variety of dimension $n\geq 2$ defined over $\Qbar$ and $G \cong \bZ^{n-1}$ a free abelian group of automorphisms of $X$ of positive entropy.
Then the following assertions hold.
\begin{enumerate}[label=\emph{(\arabic*)}, ref=(\arabic*)]
\item \label{thm:nef-big-1} There are exactly $n$ distinct multiplicative group characters $\chi_i \colon G \to (\bR_{>0}, \times)$ with $1\leq i \leq n$, where each $\chi_i$ is associated with a common nef eigendivisor $D_i$, i.e., $g^*D_i \num \chi_i(g)D_i$ for any $g\in G$.
\item \label{thm:nef-big-2} $D_1\cdots D_n \in \bR_{>0}$; in particular, $D\coloneqq \sum_{i=1}^n D_i$ is a nef and big $\bR$-divisor on $X$.
\item \label{thm:nef-big-3} For any $1\leq i\leq n$, there is some $g_i\in G$ such that $\chi_j(g_i) < 1$ for all $j\neq i$, $\chi_i(g_i) = \lambda_1(g_i)$, and $g_i^*D_i \sim_\bR \lambda_1(g_i)D_i$.
\end{enumerate}
\end{theorem}

\begin{remark}\label{rmk_DS_eigendivisor}
The construction of the common nef eigendivisors $D_i$ is essentially attributed to Dinh and Sibony, who initially constructed $n-1$ of them and then separately constructed the last one (see the last paragraph of the proof of \cite[Th\'eor\`eme~4.4]{DS04}).
This separateness is a crucial aspect, preventing a direct use of their construction to define a canonical height in an appropriate way.
Upon revisiting \cite{DS04} and examining each character \(\chi_i\colon G\to (\bR_{>0},\times)\), we demonstrate that there exist $n$ distinguished automorphisms \(g_1,\dots,g_n\) in the rank $n-1$ abelian group $G$ such that \(\chi_i(g_i)=\lambda_1(g_i)\) and \(\chi_j(g_i)<1\) for all \(j\neq i\).
Consequently, all $n$ nef eigendivisors $D_i$ share the same status as in \cref{thm:nef-big}.
It is worth mentioning that the automorphisms \(g_1,\dots,g_n\) we constructed may not necessarily be the generators of \(G\).
Instead, any \(n-1\) of them generate a finite index subgroup of \(G\) (see \cref{rmk:generators}).
\end{remark}

Before proving the above \cref{thm:nef-big} at the end of this subsection, we prepare all necessary ingredients.
Recall that $\N^1(X)_\bR$ is the real N\'eron--Severi space of $\bR$-divisors on $X$ modulo numerical equivalence $\num$.
The nef cone $\Nef(X)$, consisting of the classes of all nef $\bR$-divisors on $X$, is a salient closed convex cone in $\N^1(X)_\bR$ of full dimension. 
The pullback action of automorphisms on $\N^1(X)_\bR$ induces a natural representation of $\Aut(X)$:
\begin{align*}
\Aut(X) \to \GL(\N^1(X)_\bR), \quad g \mapsto g^*|_{\N^1(X)_\bR}.
\end{align*}
Note that any automorphism preserves the nef cone $\Nef(X) \subseteq \N^1(X)_\bR$.
For any subgroup $G$ of $\Aut(X)$, denote by $G|_{\N^1(X)_\bR} \leq \GL(\N^1(X)_\bR)$ the image of the above representation.

First of all, as a straightforward application of the previous discussion in \S\ref{subsec:commuting-family} to the triplet $(\N^1(X)_\bR, G|_{\N^1(X)_\bR}, \Nef(X))$, we get some nonzero common nef eigendivisors $D_i$ that naturally define multiplicative group characters $\chi_i$ of $G|_{\N^1(X)_\bR}$.
Composing them with the group homomorphism $G \to G|_{\N^1(X)_\bR}$ yields group characters of $G$ itself, still denoted by $\chi_i$.
In summary, we obtain the following.

\begin{proposition}
\label{prop:characters-of-G}
Let $X$ be a smooth projective variety of dimension $n\geq 2$ defined over $\Qbar$ and $G$ an abelian subgroup of $\Aut(X)$.
Then there are finitely many distinct multiplicative group characters $\chi_i \colon G \to (\bR_{>0}, \times)$ with $1\leq i \leq m$, where each $\chi_i$ is associated with a common nef eigendivisor $D_i$, i.e., $g^*D_i \num \chi_i(g)D_i$ for any $g\in G$.
Further, for any $g\in G$, there is some $1\leq i\leq m$ such that $\chi_i(g) = \rho(g^*|_{\N^1(X)_\bR}) = \lambda_1(g)$, the first dynamical degree of $g$.
\end{proposition}
\begin{proof}
It follows readily from \cref{prop:common-ev,rmk:characters,lemma:distinct-characters,lemma:chi-achieves-rho}.
\end{proof}

By the above \cref{prop:characters-of-G}, following \cite{DS04}, we define a group homomorphism
\begin{align*}
\pi \colon G & \to (\bR^{m}, +) \\
g & \mapsto (\log\chi_1(g),\dots,\log\chi_m(g)).
\end{align*}

\begin{lemma}[{cf.~\cite[Corollaire~2.2]{DS04}}]
\label{lem:dd-lower-bound}
Let \(X\) be a smooth projective variety of dimension \(n\geq 2\) defined over \(\Qbar\).
Then the set of the first dynamical degrees of surjective endomorphisms of \(X\) is discrete in \([1,+\infty)\).
\end{lemma}
\begin{proof}
For any positive number \(M>1\), we will show that the following set
\[
S_M\coloneqq\{\lambda_1(f) : f\colon X\to X, \lambda_1(f)\leq M\}
\]
is finite. 
We note that the first dynamical degree \(\lambda_1(f)\) of a surjective endomorphism \(f\colon X\to X\) is the spectral radius of \(f^*|_{\N^1(X)_\bR}\), which is induced by \(f^*|_{\N^1(X)}\).
Since \(\N^1(X)\) is a free abelian group of rank $\rho\coloneqq \rho(X)$, all the eigenvalues of \(f^*|_{\N^1(X)_\bR}\) are algebraic integers.
That is, every $\lambda_1(f)$ is the maximal modulus of the roots of a monic polynomial of degree \(\rho\) with integer coefficients.
Let \(P(t) = t^\rho + c_1t^{\rho-1} + \cdots + c_\rho\) be such a polynomial such that the maximal modulus of all roots \(\alpha_1,\dots,\alpha_\rho\), counting with multiplicities, is no more than \(M\).
Then we have 
\[
s(t)=\prod_{i=1}^\rho(t-\alpha_i).
\]
By the triangular inequality and the fact that \(|\alpha_i|\leq M\) holds for each \(i\), we obtain that each \(|c_j|\) has an upper bound.
In particular, \(\#S_M\) is finite.
We finish the proof of the lemma.
\end{proof}

\begin{proposition}[{cf.~\cite[Proposition~4.2]{DS04}}]
\label{prop:inj+discrete}
Let $G\leq \Aut(X)$ be an abelian group of automorphisms of positive entropy.
Then $\pi$ is injective and its image $\pi(G)$ is discrete in $\bR^m$.
In particular, $G$ is free abelian and $\pi(G)$ is a lattice in $\bR^m$ of rank $r \leq m$.
\end{proposition}
\begin{proof}
For any \(g\neq \id\), it follows from \cref{prop:characters-of-G} that one of the coordinates of \(\pi(g)\) coincides with \(\log\lambda_1(g)\) which is positive; hence \(\pi\) is injective. 
As \(\pi\) is a group homomorphism, to show that the image \(\pi(G)\) is discrete, it is sufficient to show that \(\pi(\id) = 0\) is an isolated point in the image \(\pi(G)\). 
Applying \cref{lem:dd-lower-bound}, we see that for each \(g\neq \id\), \(\log\lambda_1(g)\) has a uniform lower bound (which is independent of \(g\)) and hence \(\{0\}\) is an isolated point in \(\pi(G)\).
Note that the image \(\pi(G)\) is an additive subgroup which is also discrete in \(\bR^m\).
Hence \(\pi(G)\) is a lattice in \(\bR^m\).
Because \(\pi\) is injective, \(G\) is free abelian of rank \(r\leq m\).
\end{proof}

\begin{lemma}
\label{lemma:character-non-vanishing}
Let $G\leq \Aut(X)$ be a free abelian group of rank $r$ of positive entropy.
Let \(\chi_1,\dots,\chi_m\) and \(D_1,\dots,D_m\) be in \cref{prop:characters-of-G}.
Then we have
\[
D_1\cdots D_m\not\wnum 0 \text{ and } r+1 \leq m \leq \min(n, \rho(X)),
\]
where $\rho(X) \coloneqq \dim_\bR \N^1(X)_\bR$ denotes the Picard number of $X$.
\end{lemma}
\begin{proof}
We first prove that $D_{i_1}\cdots D_{i_k} \not\wnum 0$ for any $1\leq i_1 < \dots < i_k \leq \min(m,n)$ by induction on $k$.
By the higher-dimensional Hodge index theorem (see, e.g., \cite[Lemma~3.2]{Zhang16}), we have $D_i \not\wnum 0$ for each \(1\leq i\leq m\).
We suppose that the intersection product of any $j\leq k-1$ different divisors choosing from $D_1,\dots,D_m$ is not weakly numerically trivial.
Fix indices $1\leq i_1 < \dots < i_j < i_{j+1} \leq \min(m,n)$.
By inductive hypothesis, we have
\[
D_{i_1}\cdots D_{i_{j-1}} \cdot D_{i_j} \not\wnum 0 \text{ and } D_{i_1}\cdots D_{i_{j-1}} \cdot D_{i_{j+1}} \not\wnum 0.
\]
Since the $\chi_i$ are mutually distinct, $\chi_{i_j}\neq \chi_{i_{j+1}}$, i.e., there is some $g_\circ \in G$ such that $\chi_{i_j}(g_\circ)\neq \chi_{i_{j+1}}(g_\circ)$.
It follows from $G$-invariance of the $D_i$ that
\begin{align*}
g_\circ^*(D_{i_1}\cdots D_{i_{j-1}} \cdot D_{i_j}) & \wnum \chi_{i_1}(g_\circ)\cdots \chi_{i_{j-1}}(g_\circ) \, \chi_{i_j}(g_\circ) (D_{i_1}\cdots D_{i_{j-1}} \cdot D_{i_j}) \ \text{ and} \\
g_\circ^*(D_{i_1}\cdots D_{i_{j-1}}\cdot D_{i_{j+1}}) & \wnum \chi_{i_1}(g_\circ)\cdots \chi_{i_{j-1}}(g_\circ) \, \chi_{i_{j+1}}(g_\circ) (D_{i_1}\cdots D_{i_{j-1}}\cdot D_{i_{j+1}}).
\end{align*}
Noting that \(j\leq k-1\leq n-1\) by assumption, one has $D_{i_1}\cdots D_{i_j} \cdot D_{i_{j+1}} \not\wnum 0$, thanks to \cref{lemma:DS-4.4}.
In other words, we have proved by induction that the product of any $k\leq \min(m,n)$ different divisors from $D_1,\dots,D_m$ is not weakly numerically trivial.

By \cref{prop:inj+discrete,lemma:distinct-characters}, we already know that $r \leq m \leq \rho(X)$.
Suppose to the contrary that there are \(m\geq n+1\) distinct group characters \(\chi_1,\dots,\chi_{n+1},\dots,\chi_m\) of $G$ associated with some common nef eigendivisors $D_1,\dots,D_m$.
By what we just proved, one has $D_1\cdots D_n>0$ and $D_2\cdots D_{n+1}>0$.
Then it follows from the projection formula that for any \(g\in G\),
\[
\chi_1(g)\cdots\chi_n(g)=\chi_2(g)\cdots \chi_{n+1}(g) = 1,
\]
which implies that \(\chi_1=\chi_{n+1}\), a contradiction.
So we get $m\leq n$ and hence $D_1\cdots D_m\not\wnum 0$ as desired.

It remains to show that $m\geq r+1$.
Suppose to the contrary that $m=r$ so that $\pi(G)$ is a complete lattice in $\bR^m$ (see \cref{prop:inj+discrete}).
Namely, $\pi(G)$ spans $\bR^m$.
Therefore, there is some $\id \neq g_\diamond \in G$ such that all $m$ coordinates of $\pi(g_\diamond)$ are negative, i.e., $\chi_i(g_\diamond)<1$ for all $1\leq i\leq m$.
On the other hand, \cref{prop:characters-of-G} asserts that for such $g_\diamond$, there is some $1\leq i_\diamond \leq m$ such that $\chi_{i_\diamond}(g_\diamond)$ equals the first dynamical degree $\lambda_1(g_\diamond)$ of $g_\diamond$, which is strictly greater than $1$ since $g_\diamond$ is of positive entropy.
This is a contradiction.
\end{proof}

\begin{remark}
\label{rmk:R-linear}
Note that all these common nef eigendivisors $D_i$ in \cref{prop:characters-of-G} are constructed numerically.
Namely, they only satisfy eigenequations modulo numerical equivalence.
Though it suffices to define nef canonical height functions in the sense of \cite[Theorem~5(a)]{KS16}, the difference may not be bounded (see \cite[Theorem~5(b)]{KS16}).
In dimension two, Kawaguchi \cite[Lemma~3.8]{Kawaguchi08} managed to improve them to eigenequations modulo $\bR$-linear equivalence so that the difference is indeed bounded.
Such an eigenequation modulo linear equivalence also appears in \cite[Proposition~1.1.3]{Zhang06}, \cite[Lemma~2.3]{NZ10}, and \cite[Theorem~6.4(1)]{MMSZZ22}.
\end{remark}

Below is a higher-dimensional analog of \cite[Lemma~3.8]{Kawaguchi08}.

\begin{lemma}
\label{lemma:linear-equiv}
Let $X$ be a smooth projective variety of dimension $n$ over $\Qbar$ and $f$ a surjective endomorphism of $X$ of positive entropy, i.e., $\lambda_1(f)>1$.
Then the following assertions hold:
\begin{enumerate}[label=\emph{(\arabic*)}, ref=(\arabic*)]
\item \label{lemma:linear-equiv-1} there is a nef $\bR$-divisor $D_f$ on $X$ such that $f^*D_f \num \lambda_1(f)D_f$; further,
\item \label{lemma:linear-equiv-2} for any $D_f$ in the assertion~\emph{\ref{lemma:linear-equiv-1}}, there is a unique nef $\bR$-divisor $D'_f$ on $X$, up to \(\bR\)-linear equivalence, such that $D'_f \num D_f$ and $f^*D'_f \sim_\bR \lambda_1(f)D'_f$.
\end{enumerate}
\end{lemma}
\begin{proof}
It is well-known that the assertion~\ref{lemma:linear-equiv-1} follows from Birkhoff's \cref{thm:Birkhoff}.
Fix such a nef \(\bR\)-divisor $D_f\in \Nef(X)$ such that $f^*D_f \num \lambda_1(f)D_f$.
Let $\Pic^0(X)$ denote the subgroup of the Picard group $\Pic(X)$ consisting of integral divisors on $X$ algebraically equivalent to zero (modulo linear equivalence), which has the structure of an abelian variety.
Consider the exact sequence of $\bR$-vector spaces:
\[
0\to \Pic^0(X)_\bR \to \Pic(X)_\bR \to \NS(X)_\bR \cong \N^1(X)_\bR \to 0.
\]
If the irregularity $q(X)\coloneqq h^1(X, \sO_X) = 0$, then $\N^1(X)_\bR \cong \Pic(X)_\bR$ and hence the assertion~\ref{lemma:linear-equiv-2} follows.
So let us consider the case $q(X)>0$.
Note that $f^*D_f - \lambda_1(f)D_f \in \Pic^0(X)_\bR$.
\begin{claim}
\label{claim:surjective}
The $\bR$-linear map
\[
f^* \otimes_\bZ 1_\bR - \lambda_1(f) \id \colon \Pic^0(X)_\bR \to \Pic^0(X)_\bR
\]
on the (possibly infinite-dimensional) $\bR$-vector space $\Pic^0(X)_\bR$ is bijective.
\end{claim}

Assuming \cref{claim:surjective} for the time being, up to \(\bR\)-linear equivalence, there is a unique $\bR$-divisor $E\in \Pic^0(X)_\bR$ such that $f^*E-\lambda_1(f)E \sim_\bR f^*D_f - \lambda_1(f)D_f$, which yields that $f^*(D_f-E) \sim_\bR \lambda_1(f)(D_f-E)$.
Hence $D'_f \coloneqq D_f-E$ suffices to conclude the assertion~\ref{lemma:linear-equiv-2}.
\end{proof}

\begin{proof}[Proof of \cref{claim:surjective}]
Let $\pic^0_{X/\Qbar}$ denote the Picard variety of $X$.
The pullback $f^*$ of divisors on $X$ induces an isogeny $g$ of $\pic^0_{X/\Qbar}$.
Denote by $P_g(t) \in \bZ[t]$ the characteristic polynomial of $g$, which has degree $2q(X)$ and satisfies that
\begin{align*}
P_g(t) & = \det(t \id - g^* \, \rotatebox[origin=c]{-90}{$\circlearrowright$} \ H^1(\pic^0_{X/\Qbar}, \bQ)) \\
& = \det(t \id - f_* \, \rotatebox[origin=c]{-90}{$\circlearrowright$} \ H^{2n-1}(X, \bQ)) \\
& = \det(t \id - f^* \, \rotatebox[origin=c]{-90}{$\circlearrowright$} \ H^{1}(X, \bQ)),
\end{align*}
where $H^{2n-1}(X, \bQ))$ is canonically isomorphic to $H^1(\pic^0_{X/\Qbar}, \bQ))$ via the Poincar\'e divisor on $X\times \pic^0_{X/\Qbar}$ and the last equality follows from Poincar\'e duality.
Thanks to \cite[Proposition~5.8]{Dinh05}, all roots of $P_g(t)$ have moduli at most $\sqrt{\lambda_1(f)}$.
In particular, $P_g(\lambda_1(f)) \neq 0$ since $\lambda_1(f)>1$ by assumption.
Besides, since the rational representation $\End(\pic^0_{X/\Qbar}) \to \End_\bQ(H_1(\pic^0_{X/\Qbar}, \bQ))$ is an injective homomorphism, $P_g(g) = 0$ as an endomorphism of $\pic^0_{X/\Qbar}$.
It follows that $P_g(f^*) = 0$ as a group homomorphism of $\Pic^0(X) = \pic^0_{X/\Qbar}(\Qbar)$.
Now by applying \cref{lem-operator-infinite-dim} to the $\bZ$-module $\Pic^0(X)$, \cref{claim:surjective} follows.
\end{proof}

Now we are ready to prove \cref{thm:nef-big}.

\begin{proof}[Proof of \cref{thm:nef-big}]
Assertions~\ref{thm:nef-big-1} and \ref{thm:nef-big-2} follow readily from \cref{prop:characters-of-G}, \cref{lemma:character-non-vanishing}, and the assumption that $\rank G = n-1$.
Let $1\leq i \leq n$ be fixed.
Let $p_i$ denote the natural projection from $\bR^n$ to $\bR^{n-1}$ by omitting the $i$-th coordinate $x_i$.
Recall that the group homomorphism $\pi\colon G \to (\bR^n, +)$, defined by sending $g\in G$ to $(\log\chi_1(g),\dots,\log\chi_n(g))$, is injective and $\pi(G)$ is a lattice in $\bR^n$ of rank $n-1$ (see \cref{prop:inj+discrete}).
Besides, it follows from the projection formula that for any $g\in G$, $\chi_1(g) \cdots \chi_n(g)=1$.
Hence the image $\pi(G)$ of $G$ is actually contained in the hyperplane 
$H \coloneqq \{(x_1,\dots,x_n) \in \bR^n : \sum_{j=1}^n x_j=0\}$.

Consider the following commutative diagram
\begin{equation*}
\begin{tikzcd}
G \cong \bZ^{n-1} \arrow[r, hook, "\pi"] \arrow[d, hook, "\tau"'] & \bR^n \arrow[d, two heads, "p_i"] \\
H \arrow[r, "p_i \circ \iota"] \arrow[ru, hook, "\iota"] & \bR^{n-1}.
\end{tikzcd}
\end{equation*}
Clearly, $p_i \circ \iota\colon H \to \bR^{n-1}$ is an isomorphism of $\bR$-vector spaces.
By the open mapping theorem, it is also an isomorphism of topological vector spaces.
Denote $p_i \circ \pi$ by $\pi_i$. 
Since $\tau(G)$ is a lattice in $H$ of rank $n-1$ and \(p_i\circ\iota\) is a topological isomorphism, $\pi_i(G) = (p_i \circ \iota)(\tau(G))$ is a lattice in $\bR^{n-1}$ of rank $n-1$.
Therefore, there is some $\id \neq g_i \in G$ such that all $n-1$ coordinates of $\pi_i(g_i)$ are negative, i.e., $\chi_j(g_i)<1$ for all $j\neq i$.
Further, by \cref{prop:characters-of-G}, for such $g_i$, there is some $1\leq t_i \leq n$ such that $\chi_{t_i}(g_i) = \lambda_1(g_i)>1$.
Clearly, $t_i$ has to be $i$.
We thus proved that $g_i^*D_i \num \lambda_1(g_i)D_i$.
By \cref{lemma:linear-equiv}\ref{lemma:linear-equiv-2}, there is a nef $\bR$-divisor $D'_i \num D_i$ such that $g_i^*D'_i \sim_\bR \lambda_1(g_i)D'_i$.

In the end, we replace $D_i$ with $D'_i$ for all $i$.
Since $D'_i \num D_i$, this does not affect Assertions~\ref{thm:nef-big-1} and \ref{thm:nef-big-2}.
We thus complete the proof of \cref{thm:nef-big}.
\end{proof}

\begin{remark}
\label{rmk:generators}
In the above \cref{thm:nef-big}, denote by $M$ the matrix $(\log\chi_i(g_j))_{1\leq i,j\leq n}$.
Clearly, the rank of $M$ is at most $n-1$ since $\sum_{i=1}^n \log\chi_i = 0$.
On the other hand, as $\log\chi_i(g_j)<1$ for all $i\neq j$, one can see that any submatrix $M_i$ of $M$ obtained by deleting the $i$-th row and the $i$-th column is strictly diagonally dominant and hence nonsingular (see \cite[Theorem~6.1.10]{HJ13}).
It follows that $\rank M=n-1$.
Hence $g_1,\dots,g_n$ generate a free abelian subgroup of $G$ of full rank $n-1$;
moreover, by the argument, so do any $n-1$ automorphisms from $g_1,\dots,g_n$.
\end{remark}

We end this section with the following lemma, which will be used later in the proof of \cref{thm:A} (or rather, \cref{thm:A-full}, and its corollaries).

\begin{lemma}
\label{lemma:aug-base-loci-G-inv}
With notation as in \cref{thm:nef-big}, the augmented base locus $\bB_+(D)$ of the nef and big $\bR$-divisor $D$ is a $G$-invariant Zariski closed proper subset of $X$.
In particular, for any $x\in (X\backslash \bB_+(D))(\Qbar)$, one has $\cO_g(x)\cap \bB_+(D) = \emptyset$ for any $g\in G$.
\end{lemma}
\begin{proof}
It is well-known that $\bB_+(D) \neq X$ if and only if $D$ is big; see \cref{prop:B_+}\ref{prop:B_+-big-iff-nx}.
Thanks to an observation by Lesieutre and Satriano (see \cref{lemma:LS-2.16}), for any positive numbers \(a_1,\dots,a_n\), we have
\[
\bB_+(a_1D_1+\dots+a_nD_n)=\bB_+(D_1+\dots+D_n).
\]
It follows from \cref{prop:B_+}\ref{prop:B_+-num} that for any \(g\in G\),
\[
g(\bB_+(D)) = \bB_+((g^{-1})^*(D_1+\cdots+D_n)) = \bB_+\bigg(\sum_{i=1}^n\chi_i(g)^{-1}D_i\bigg) = \bB_+(D).
\]
Therefore, \(\bB_+(D)\) is \(g\)-invariant.
\end{proof}

\section{Canonical heights for abelian group actions}
\label{section:can-ht}

Throughout this section, $X$ is a smooth projective variety of dimension $n\geq 2$ defined over $\Qbar$ and $G \cong \bZ^{n-1}$ is a free abelian group of automorphisms of $X$ of positive entropy.
For brevity, we also assume the following hypothesis; thanks to \cref{thm:nef-big}, it always holds true.

\begin{hypothesis}
\label{key-hyp}
There are $n$ distinct group characters $\chi_1,\dots,\chi_n$ of $G$ associated with some common nef eigendivisors $D_1,\dots,D_n$ on $X$ and $n$ automorphisms $g_1,\dots,g_{n} \in G$ such that
\begin{enumerate}
\item for any $g\in G$ and any $1\leq i\leq n$,
\[
g^*D_i \num \chi_i(g)D_i;
\]
\item $D\coloneqq D_1+\cdots+D_n$ is a nef and big $\bR$-divisor on $X$;
\item for any $1\leq i\leq n$,
\[
g_i^*D_i \sim_\bR \lambda_1(g_i)D_i.
\]
\end{enumerate}
\end{hypothesis}

We shall construct a canonical height function $\widehat{h}_G$ associated with $G$.

\begin{lemma}
\label{lemma:nef-can-ht}
Under \cref{key-hyp}, for any $1\leq i\leq n$ and any $x\in X(\Qbar)$, the limit
\[
\widehat{h}_{D_i,g_i}(x) \coloneqq \lim_{m\to \infty} \frac{h_{D_i}(g_i^m(x))}{\lambda_1(g_i)^m}
\]
exists and satisfies the following properties.
\begin{enumerate}[label=\emph{(\arabic*)}, ref=(\arabic*)]
\item \label{lemma:nef-can-ht-1} $\widehat{h}_{D_i,g_i} = h_{D_i} + O(1)$.

\item \label{lemma:nef-can-ht-2} $\widehat{h}_{D_i,g_i}\circ g = \chi_i(g) \, \widehat{h}_{D_i,g_i}$ for any $g\in G$; in particular, $\widehat{h}_{D_i,g_i}\circ g_i = \lambda_1(g_i) \, \widehat{h}_{D_i,g_i}$.
\end{enumerate}
\end{lemma}
\begin{proof}
Note that $g_i^*D_i \sim_\bR \lambda_1(g_i)D_i$ by \cref{thm:nef-big}\ref{thm:nef-big-3}.
Hence the existence of each \(\widehat{h}_{D_i,g_i}\) and the property~\ref{lemma:nef-can-ht-1} follow immediately from \cref{thm:nef-can-ht}.

For the property~\ref{lemma:nef-can-ht-2}, fix an integer \(i\) with $1\leq i\leq n$, an automorphism $g\in G$ of positive entropy, an ample divisor $H_X$ on $X$, and a height function $h_{H_X}$ associated with $H_X$.
Then thanks to Matsuzawa \cite[Theorem~1.7(2)]{Matsuzawa20}, there is a constant $C_1>0$ such that for any rational point $x\in X(\Qbar)$ and any $m\geq 1$,
\begin{align}
\label{eq:Matsuzawa}
h_{H_X}^+(g_i^m(x)) \leq C_1 m^{\rho(X)-1}\lambda_1(g_i)^m h_{H_X}^+(x).
\end{align}
Observe that by \cref{thm:nef-big}\ref{thm:nef-big-1}, we have \(g^*D_i \num \chi_i(g)D_i\).
Then according to \cref{thm:weil-height}\ref{thm:weil-height-functoriality} (Functoriality) and \ref{thm:weil-height-alg-equiv} (Algebraic equivalence), there is a constant $C_2>0$ such that
\begin{align*}
|h_{g^*D_i} - h_{D_i}\circ g| & \leq C_2, \\
|h_{g^*D_i} - \chi_i(g) \, h_{D_i}| & \leq C_2 \sqrt{h_{H_X}^+}.
\end{align*}
Clearly, combining them together yields that
\[
|h_{D_i}\circ g - \chi_i(g) \, h_{D_i}| \leq 2C_2 \sqrt{h_{H_X}^+}.
\]
In particular, for any $m\geq 1$, one has
\begin{equation}
\label{eq:difference}
|h_{D_i}(g(g_i^m(x))) - \chi_i(g) \, h_{D_i}(g_i^m(x))| \leq 2C_2 \sqrt{h_{H_X}^+(g_i^m(x))}. 
\end{equation}
As $G$ is abelian, it follows from \cref{eq:Matsuzawa,eq:difference} that
\begin{align*}
\left|\frac{h_{D_i}(g_i^m(g(x)))}{\lambda_1(g_i)^m}-\frac{\chi_i(g)h_{D_i}(g_i^m(x))}{\lambda_1(g_i)^m}\right| & \leq \frac{2C_2 \sqrt{h_{H_X}^+(g_i^m(x))}}{\lambda_1(g_i)^m} \\
& \leq C_3 \sqrt{\frac{m^{\rho(X)-1}}{\lambda_1(g_i)^{m}}},
\end{align*}
where $C_3>0$ is a constant independent of $m\geq 1$.
Note that both limits of the left-hand side exist (see \cref{thm:nef-can-ht}).
So taking $m\to \infty$ yields that $\widehat{h}_{D_i,g_i}(g(x)) = \chi_i(g) \, \widehat{h}_{D_i,g_i}(x)$.
\end{proof}

The theorem below is a precise version of our main result \cref{thm:A}.
Indeed, one just takes the Zariski closed proper subset $Z$ to be the augmented base locus $\bB_+(D)$ of the nef and big $\bR$-divisor $D$ constructed in \cref{thm:nef-big}.
It extends \cite[Theorem~1.1]{Silverman91} and \cite[Theorem~5.2]{Kawaguchi08} to higher dimensions (under the maximal dynamical rank assumption).

\begin{theorem}
\label{thm:A-full}
Under \cref{key-hyp}, for each $1\leq i\leq n$, let $\widehat{h}_{D_i,g_i}$ be in \cref{lemma:nef-can-ht}.
Define a function $\widehat{h}_G \colon X(\Qbar)\to \bR$ by
\[
\widehat{h}_G \coloneqq \sum_{i=1}^n \widehat{h}_{D_i,g_i}.
\]
Then the following assertions hold.
\begin{enumerate}[label=\emph{(\arabic*)}, ref=(\arabic*)]
\item \label{thm:A-full-def} The function \(\widehat{h}_{G}\) is a Weil height corresponding to the nef and big $\bR$-divisor \(D\), i.e.,
\[
\widehat{h}_{G} = h_{D} + O(1).
\]

\item \label{thm:A-full-functorial} For any $g\in G$, one has
\[
\widehat{h}_{G} \circ g = \sum_{i=1}^n \chi_i(g) \, \widehat{h}_{D_i,g_i}.
\]

\item \label{thm:A-full-positivity} For any $1\leq i\leq n$ and any \(x\in (X\backslash\bB_+(D))(\Qbar)\), one has \(\widehat{h}_{D_i,g_i}(x)\geq 0\) and hence \(\widehat{h}_{G}(x)\geq 0\), where $\bB_+(D)$ is the augmented base locus of $D$ (see \cref{lemma:aug-base-loci-G-inv}).

\item \label{thm:A-full-Northcott} 
The Weil height function \(\widehat{h}_{G}\) satisfies the Northcott finiteness property on \(X\backslash\bB_+(D)\); in other words, for any positive integer \(d\) and real number \(N\), the set
\[
\big\{x \in (X\backslash\bB_+(D))(\Qbar) : [\bQ(x):\bQ]\leq d,~\widehat{h}_{G}(x)\leq N \big\}
\]
is finite.

\item \label{thm:A-full-zero-loci} For any \(x\in (X\backslash\bB_+(D))(\Qbar)\), the following statements are equivalent.
\begin{enumerate}[label=\emph{(\roman*)}, ref=(\roman*)]
\item \label{thm:A-full-zero-loci-3} \(\widehat{h}_{G}(x)=0\).

\item \label{thm:A-full-zero-loci-2} \(\widehat{h}_{D_i,g_i}(x)=0\) for all \(1\leq i\leq n\).

\item \label{thm:A-full-zero-loci-1} \(x\) is \(g\)-periodic for any \(g\in G\).

\item \label{thm:A-full-zero-loci-5} \(x\) is \(g\)-periodic for some $\id \neq g\in G$.

\item \label{thm:A-full-zero-loci-4} \(\widehat{h}_{D_i,g_i}(x)=0\) for some \(1\leq i\leq n\).
\end{enumerate}
\end{enumerate}
We call $\widehat{h}_G$ \emph{a canonical height function} associated with the abelian group $G$ of maximal dynamical rank.
\end{theorem}
\begin{proof}
Assertion~\ref{thm:A-full-def} follows immediately from \cref{lemma:nef-can-ht}\ref{lemma:nef-can-ht-1} and \cref{thm:weil-height}\ref{thm:weil-height-additivity} (Additivity).
Assertion~\ref{thm:A-full-functorial} follows from \cref{lemma:nef-can-ht}\ref{lemma:nef-can-ht-2}.
Since $D$ is a big $\bR$-divisor, the height function $h_D$ satisfies the Northcott finiteness property (see \cref{thm:Northcott}), so does $\widehat{h}_G$ by Assertion~\ref{thm:A-full-def}.
We have thus proved Assertion~\ref{thm:A-full-Northcott}.

Next, we shall show Assertion~\ref{thm:A-full-positivity}.
Fix an index $i$ with \(1\leq i\leq n\).
According to \cref{lemma:aug-base-loci-G-inv}, the augmented base locus $\bB_+(D)$ of $D$ is a $G$-invariant Zariski closed proper subset of $X$.
Fix a point $x\in (X\backslash\bB_+(D))(\Qbar)$; in particular, $\cO_{g_i}(x)\cap \bB_+(D) = \emptyset$.
We notice by \cite[Lemma~2.26]{LS21} that $h_D \geq O(1)$ outside $\bB_+(D)$. 
On the other hand, by Assertion~\ref{thm:A-full-def}, we have
\[
h_D = \widehat{h}_G + O(1) = \sum_{j=1}^n \widehat{h}_{D_j,g_j} + O(1).
\]
It follows that for any $m\geq 1$,
\begin{align}\label{eq-thm:A-full-bounded}
O(1) \leq \sum_{j=1}^n \widehat{h}_{D_j,g_j}(g_i^m(x)) = \sum_{j=1}^n \chi_j(g_i)^m \, \widehat{h}_{D_j,g_j}(x),
\end{align}
where the equality is from Assertion~\ref{thm:A-full-functorial}.
Note that $\chi_j(g_i) < 1$ for all $j\neq i$ and $\chi_i(g_i) = \lambda_1(g_i)>1$ by the choice of \(g_i\); see \cref{thm:nef-big}\ref{thm:nef-big-3}.
Dividing \cref{eq-thm:A-full-bounded} by \(\chi_i(g_i)^m\) from both sides and letting \(m\) tend to infinity, it is easy to see that $\widehat{h}_{D_i,g_i}(x)$ has to be nonnegative.
This thus shows Assertion~\ref{thm:A-full-positivity}.

At last, we prove Assertion~\ref{thm:A-full-zero-loci}.
Fix a rational point $x\in (X\backslash\bB_+(D))(\Qbar)$; in particular, $\cO_{g}(x)\cap \bB_+(D) = \emptyset$ for any $g\in G$.
We shall prove the equivalence in the following order:
\[
\begin{tikzcd}
\textup{\ref{thm:A-full-zero-loci-3}} \arrow[r, Leftrightarrow] & \textup{\ref{thm:A-full-zero-loci-2}} \arrow[r, Rightarrow] & \textup{\ref{thm:A-full-zero-loci-1}} \arrow[d, Rightarrow] \\
 & \textup{\ref{thm:A-full-zero-loci-4}} \arrow[u, Rightarrow] & \textup{\ref{thm:A-full-zero-loci-5}} \arrow[l, Rightarrow].
\end{tikzcd}
\]
By definition, \ref{thm:A-full-zero-loci-2} $\Rightarrow$ \ref{thm:A-full-zero-loci-3} is trivial.
\ref{thm:A-full-zero-loci-3} $\Rightarrow$ \ref{thm:A-full-zero-loci-2} follows from Assertion~\ref{thm:A-full-positivity}.
\ref{thm:A-full-zero-loci-1} $\Rightarrow$ \ref{thm:A-full-zero-loci-5} is also trivial.
We shall first prove \ref{thm:A-full-zero-loci-2} $\Rightarrow$ \ref{thm:A-full-zero-loci-1}.
Let $g\in G$ be fixed.
By Assertion~\ref{thm:A-full-functorial}, for any $m\geq 1$,
\[
\widehat{h}_{G}(g^m(x)) = \sum_{i=1}^n \chi_i(g)^m \, \widehat{h}_{D_i,g_i}(x) = 0.
\]
Since these rational points $g^m(x)$ are of bounded degree over $\bQ$, it follows from Assertion~\ref{thm:A-full-Northcott} that the forward $g$-orbit $\cO_g(x)$ of $x$ is finite, i.e., $x$ is $g$-periodic, noting that \(g\) is an automorphism.

We next show \ref{thm:A-full-zero-loci-5} $\Rightarrow$ \ref{thm:A-full-zero-loci-4}.
Suppose that $x$ is $g$-periodic for some $\id \neq g \in G$.
Note that by the construction of the characters $\chi_i$ (see \cref{prop:characters-of-G}), there is some $1\leq i\leq n$ such that $\chi_i(g) = \lambda_1(g)$.
Consider the growth of the function $\widehat{h}_G$ along the forward $g$-orbit $\cO_g(x)$ of $x$, which is finite by assumption.
In other words, we have
\[
O(1) = \widehat{h}_{G}(g^m(x)) \geq \lambda_1(g)^m \, \widehat{h}_{D_i,g_i}(x),
\]
where the last inequality is due to Assertion~\ref{thm:A-full-positivity}. 
As $\lambda_1(g)>1$, we see that $\widehat{h}_{D_i,g_i}(x)$ has to be zero by letting \(m\) tend to infinity. 

It remains to prove \ref{thm:A-full-zero-loci-4} $\Rightarrow$ \ref{thm:A-full-zero-loci-2}.
Without loss of generality, we may assume that $\widehat{h}_{D_1,g_1}(x) = 0$. 
Via a similar argument in the proof of \ref{thm:A-full-zero-loci-2} $\Rightarrow$ \ref{thm:A-full-zero-loci-1}, we claim that $x$ is $g_1$-periodic.
Indeed, by Assertion~\ref{thm:A-full-functorial} and \cref{thm:nef-big}\ref{thm:nef-big-3}, we have for any $m\geq 1$,
\[
\widehat{h}_{G}(g_1^m(x)) = \sum_{j=2}^n \chi_j(g_1)^m \, \widehat{h}_{D_j}(x) < \sum_{j=2}^n \widehat{h}_{D_j}(x).
\]
Note that the $g_1^m(x)$ are of bounded degree over $\bQ$.
By Assertion~\ref{thm:A-full-Northcott}, $x$ is $g_1$-periodic.
Denote the finite period by $e_1\in \bZ_{>0}$, i.e., $g_1^{e_1}(x)= x$.
Consider the growth of the function $\widehat{h}_{G}$ along the orbit $\cO_{g_1^{-e_1}}(x)$ of $x$ under the automorphism $g_1^{-e_1}$.
Precisely, for any $m\geq 1$, we have
\[
\widehat{h}_{G}(x) = \widehat{h}_{G}(g_1^{-e_1}(x)) = \widehat{h}_{G}(g_1^{-e_1m}(x)) = \sum_{j=2}^n \chi_j(g_1)^{-e_1m} \, \widehat{h}_{D_j}(x), 
\]
which forces $\widehat{h}_{D_j}(x) = 0$ for all $j>1$ 
by noting that $\chi_j(g_1)<1$. 
This verifies \ref{thm:A-full-zero-loci-4} $\Rightarrow$ \ref{thm:A-full-zero-loci-2}. 

We thus complete the proof of \cref{thm:A-full}.
\end{proof}

As a direct consequence of \cref{thm:A-full}, we also obtain the following.

\begin{corollary}[{cf.~\cite[Proposition~7 and Proof of Theorem~2(c)]{KS14-exa}}]
\label{cor:non-periodic}
Under \cref{key-hyp}, for any \(g\in G\) and any point $x\in (X\backslash\bB_+(D))(\Qbar)$, we have
\[
\alpha_g(x) = \left\{
\begin{array}{cl}
1 & \text{if $x$ is $g$-periodic}, \\
\lambda_1(g) & \text{if $x$ is not $g$-periodic}.
\end{array}
\right.
\]
\end{corollary}
\begin{proof}
Let $g\in G$ be fixed.
First, we assume that $x\in (X\backslash\bB_+(D))(\Qbar)$ is a non-$g$-periodic point.
As we mentioned before, the limit defining $\alpha_g(x)$ exists and is independent of the choice of the ample divisor (see \cite[Theorem~3]{KS16-abelian} and \cite[Proposition~12]{KS16}, respectively).
Choose an ample divisor $H_X$ such that $H_X - D$ is ample (noting that the ample cone is open).
It follows from \cref{thm:weil-height}\ref{thm:weil-height-additivity} (Additivity) and \ref{thm:weil-height-positivity} (Positivity) that for any $m\geq 1$,
\[
h_{H_X}(g^m(x)) = h_{D}(g^m(x)) + h_{H_X-D}(g^m(x)) + O(1) \geq h_{D}(g^m(x)) + O(1).
\]
On the other hand, \cref{thm:A-full}\ref{thm:A-full-def} asserts that
\[
h_{D}(g^m(x)) = \widehat{h}_G(g^m(x)) + O(1).
\]
Putting them together yields that
\[
h_{H_X}(g^m(x)) \geq \widehat{h}_G(g^m(x)) + O(1) = \sum_{j=1}^n \chi_j(g)^m \, \widehat{h}_{D_j,g_j}(x) + O(1),
\]
where the equality is from \cref{thm:A-full}\ref{thm:A-full-functorial}.
Furthermore, by \cref{prop:characters-of-G}, there is some \(1\leq i\leq n\) such that \(\chi_i(g) = \lambda_1(g)\).
We thus obtain that
\[
h_{H_X}(g^m(x)) \geq \lambda_1(g)^m \, \widehat{h}_{D_i,g_i}(x) + O(1).
\]
Note that the term $O(1)$ does not depend on $x$ nor $m$.
Also, according to \cref{thm:A-full}\ref{thm:A-full-zero-loci}, one has \(\widehat{h}_{D_i,g_i}(x) > 0\) for all \(1\leq i\leq n\).
Now by taking $m$-th roots and letting $m\to \infty$, one easily has $\alpha_g(x) \geq \lambda_1(g)$.
The reverse inequality is due to \cite[Theorem~4]{KS16} or \cite[Theorem~1.4]{Matsuzawa20}.

Second, we assume that \(x\) is \(g\)-periodic. 
Then with the ample divisor \(H_X\) chosen as above, it is clear that \((h_{H_X}(g^m(x)))_{m\in\bN}\) is a finite set. Therefore,
\[
1\leq \alpha_g(x)=\lim_{m\to\infty}(h_{H_X}^+(g^m(x)))^{1/m}\leq 1.
\]
We finish the proof of \cref{cor:non-periodic}.
\end{proof}

\begin{remark}\label{rmk:H_G}
Under \cref{key-hyp}, following \cite{Silverman91}, we can also define a function
\[
\widehat{H}_G\colon X(\Qbar) \to \bR \quad \text{by} \quad 
\widehat{H}_G(x) = \prod_{i=1}^n \widehat{h}_{D_i,g_i}(x).
\]
Note that by the projection formula, for any $g\in G$, one has $\prod_{i=1}^n \chi_i(g) = 1$.
It follows that
\[
\widehat{H}_G \circ g = \prod_{i=1}^n \widehat{h}_{D_i,g_i} \circ g = \prod_{i=1}^n \chi_i(g) \, \widehat{h}_{D_i,g_i} = \prod_{i=1}^n \widehat{h}_{D_i,g_i} = \widehat{H}_G.
\]
In other words, the function $\widehat{H}_G$ is $G$-invariant.
Let $x\in (X\backslash\bB_+(D))(\Qbar)$ be arbitrary.
Then according to \cref{thm:A-full}\ref{thm:A-full-positivity}, each \(\widehat{h}_{D_i,g_i}(x)\geq 0\).
This yields that
\[
\sqrt[\leftroot{-1}\uproot{6} n]{\widehat{H}_G(x)} \leq \frac{\widehat{h}_G(x)}{n}.
\]
Moreover, by \cref{thm:A-full}\ref{thm:A-full-zero-loci}, \(\widehat{H}_G(x)=0\) if and only if \(x\) is $g$-periodic for any \(g\in G\).
\end{remark}

Inspired by \cite[Theorem~1.3(a)]{Silverman91}, we ask the following.

\begin{question}
\label{question:H_G}
For any point $x\in (X\backslash\bB_+(D))(\Qbar)$ with infinite $G$-orbit, is there any lower bound of $\widehat{H}_G(x)$ in terms of $\widehat{h}_G(x)$?
\end{question}

Note that if the above question has an affirmative answer, then one could prove a similar result as \cite[Theorem~1.2(b)]{Silverman91} using the Northcott property for $\widehat{h}_G$ (i.e., \cref{thm:A-full}\ref{thm:A-full-Northcott}).
Namely, there are only finitely many infinite $G$-orbits in $(X\backslash\bB_+(D))(K)$, where $K$ is any number field.

\subsection{Proofs of Theorem \ref{thm:A} and Corollaries \ref{cor:KSC} and \ref{cor:arithmetic-estimate}}

\begin{proof}[Proof of \cref{thm:A}]
By \cref{thm:nef-big}, we can always assume \cref{key-hyp}.
Take \(Z\) to be the augmented base locus \(\bB_+(D)\) of \(D\), which is a $G$-invariant Zariski closed proper subset of $X$ (see \cref{lemma:aug-base-loci-G-inv}).
\cref{thm:A} then follows easily from \cref{thm:A-full}.
\end{proof}

\begin{proof}[Proof of \cref{cor:KSC}]
It follows readily from \cref{lemma:aug-base-loci-G-inv,cor:non-periodic}.
\end{proof}

\begin{proof}[Proof of \cref{cor:arithmetic-estimate}]
Take \(Z\) to be the augmented base locus \(\bB_+(D)\) of \(D\) as in \cref{thm:A}.
Fix a $g$-periodic point $x\in (X\backslash Z)(\Qbar)$ and an ample divisor \(H_X\) on \(X\).
By \cref{thm:A-full}\ref{thm:A-full-def} and \ref{thm:A-full-zero-loci}, we have $h_D(x) = \widehat{h}_G(x) + O(1) = O(1)$. 
Besides, according to \cref{prop:B_+}\ref{prop:B_+-modification}, there is \(\varepsilon>0\) such that \(D-\varepsilon H_X\) is an effective $\bQ$-divisor and \(\bB_+(D) = \bB(D-\varepsilon H_X) = \Bs(M(D-\varepsilon H_X))\) for some $M\geq 1$.
Since \(x\not\in \bB_+(D)\), by applying \cref{thm:weil-height}\ref{thm:weil-height-additivity} (Additivity) and \ref{thm:weil-height-positivity} (Positivity) to \(M(D-\varepsilon H_X)\), we obtain that \(h_{D-\varepsilon H_X}(x)\geq O(1)\).
It thus follows that
\begin{align*}
O(1) = h_D(x) = h_{D-\varepsilon H_X}(x)+h_{\varepsilon H_X}(x)+O(1) \geq h_{\varepsilon H_X}(x) + O(1).
\end{align*}
Therefore, \(h_{\varepsilon H_X}(x)\) and \(h_{H_X}(x)\) are both bounded.
Assertion~\ref{cor:arithmetic-estimate1} is thus proved.

Assertion~\ref{cor:arithmetic-estimate2} follows easily from \cite[Proposition~3]{KS16} and \cref{cor:non-periodic}.
\end{proof}

\bibliographystyle{amsalpha}
\bibliography{KSC}

\newcommand{\etalchar}[1]{$^{#1}$}
\providecommand{\bysame}{\leavevmode\hbox to3em{\hrulefill}\thinspace}
\providecommand{\MR}{\relax\ifhmode\unskip\space\fi MR }
\providecommand{\MRhref}[2]{%
  \href{http://www.ams.org/mathscinet-getitem?mr=#1}{#2}
}
\providecommand{\href}[2]{#2}
\begin{thebibliography}{MMS{\etalchar{+}}22}

\bibitem[Bir67]{Birkhoff67}
Garrett Birkhoff, \emph{Linear transformations with invariant cones}, Amer. Math. Monthly \textbf{74} (1967), 274--276. \MR{0214605}

\bibitem[CS93]{CS93}
Gregory~S. Call and Joseph~H. Silverman, \emph{Canonical heights on varieties with morphisms}, Compositio Math. \textbf{89} (1993), no.~2, 163--205. \MR{1255693}

\bibitem[CWZ14]{CWZ14}
Frederic Campana, Fei Wang, and De-Qi Zhang, \emph{Automorphism groups of positive entropy on projective threefolds}, Trans. Amer. Math. Soc. \textbf{366} (2014), no.~3, 1621--1638. \MR{3145744}

\bibitem[Dan20]{Dang20}
Nguyen-Bac Dang, \emph{Degrees of iterates of rational maps on normal projective varieties}, Proc. Lond. Math. Soc. (3) \textbf{121} (2020), no.~5, 1268--1310. \MR{4133708}

\bibitem[DGH{\etalchar{+}}22]{DGHLS22}
Nguyen-Bac Dang, Dragos Ghioca, Fei Hu, John Lesieutre, and Matthew Satriano, \emph{Higher arithmetic degrees of dominant rational self-maps}, Ann. Sc. Norm. Super. Pisa Cl. Sci. (5) \textbf{23} (2022), no.~1, 463--481. \MR{4407198}

\bibitem[DHZ15]{DHZ15}
Tien-Cuong Dinh, Fei Hu, and De-Qi Zhang, \emph{Compact {K}\"ahler manifolds admitting large solvable groups of automorphisms}, Adv. Math. \textbf{281} (2015), 333--352. \MR{3366842}

\bibitem[Din05]{Dinh05}
Tien-Cuong Dinh, \emph{Suites d'applications m\'{e}romorphes multivalu\'{e}es et courants laminaires}, J. Geom. Anal. \textbf{15} (2005), no.~2, 207--227. \MR{2152480}

\bibitem[Din12]{Din12}
\bysame, \emph{Tits alternative for automorphism groups of compact {K}\"{a}hler manifolds}, Acta Math. Vietnam. \textbf{37} (2012), no.~4, 513--529. \MR{3058661}

\bibitem[DS04]{DS04}
Tien-Cuong Dinh and Nessim Sibony, \emph{Groupes commutatifs d'automorphismes d'une vari\'et\'e k\"ahl\'erienne compacte}, Duke Math. J. \textbf{123} (2004), no.~2, 311--328. \MR{2066940}

\bibitem[DS05]{DS05}
\bysame, \emph{Une borne sup\'{e}rieure pour l'entropie topologique d'une application rationnelle}, Ann. of Math. (2) \textbf{161} (2005), no.~3, 1637--1644. \MR{2180409}

\bibitem[ELM{\etalchar{+}}06]{ELMNP06}
Lawrence Ein, Robert Lazarsfeld, Mircea Musta\c{t}\u{a}, Michael Nakamaye, and Mihnea Popa, \emph{Asymptotic invariants of base loci}, Ann. Inst. Fourier (Grenoble) \textbf{56} (2006), no.~6, 1701--1734. \MR{2282673}

\bibitem[Gro03]{Gromov03}
Mikha{\"{\i}}l Gromov, \emph{On the entropy of holomorphic maps}, Enseign. Math. (2) \textbf{49} (2003), no.~3-4, 217--235, preprint SUNY (1977). \MR{2026895}

\bibitem[HJ13]{HJ13}
Roger~A. Horn and Charles~R. Johnson, \emph{Matrix analysis}, second ed., Cambridge University Press, Cambridge, 2013. \MR{2978290}

\bibitem[HS00]{GTM201}
Marc Hindry and Joseph~H. Silverman, \emph{Diophantine geometry. {A}n introduction}, Graduate Texts in Mathematics, vol. 201, Springer-Verlag, New York, 2000. \MR{1745599}

\bibitem[Hu20]{Hu20}
Fei Hu, \emph{A theorem of {T}its type for automorphism groups of projective varieties in arbitrary characteristic}, Math. Ann. \textbf{377} (2020), no.~3-4, 1573--1602, With an appendix by Tomohide Terasoma. \MR{4126902}

\bibitem[Hu23]{Hu-lc}
\bysame, \emph{Eigenvalues and dynamical degrees of self-maps on abelian varieties}, J. Algebraic Geom. (2023+), 29 pp., electronically published on February 7, 2023, \doi{10.1090/jag/806}.

\bibitem[Kaw06a]{Kawaguchi06}
Shu Kawaguchi, \emph{Canonical height functions for affine plane automorphisms}, Math. Ann. \textbf{335} (2006), no.~2, 285--310. \MR{2221115}

\bibitem[Kaw06b]{Kawaguchi06-Crelle}
\bysame, \emph{Canonical heights, invariant currents, and dynamical eigensystems of morphisms for line bundles}, J. Reine Angew. Math. \textbf{597} (2006), 135--173. \MR{2264317}

\bibitem[Kaw08]{Kawaguchi08}
\bysame, \emph{Projective surface automorphisms of positive topological entropy from an arithmetic viewpoint}, Amer. J. Math. \textbf{130} (2008), no.~1, 159--186. \MR{2382145}

\bibitem[KK01]{KK01}
Boris Kalinin and Anatole Katok, \emph{Invariant measures for actions of higher rank abelian groups}, Smooth ergodic theory and its applications ({S}eattle, {WA}, 1999), Proc. Sympos. Pure Math., vol.~69, Amer. Math. Soc., Providence, RI, 2001, pp.~593--637. \MR{1858547}

\bibitem[KKRH14]{KKRH14}
Anatole Katok, Svetlana Katok, and Federico Rodriguez~Hertz, \emph{The {F}ried average entropy and slow entropy for actions of higher rank abelian groups}, Geom. Funct. Anal. \textbf{24} (2014), no.~4, 1204--1228. \MR{3248484}

\bibitem[Kle05]{Kle05}
Steven~L. Kleiman, \emph{The {P}icard scheme}, Fundamental algebraic geometry: {G}rothendieck's {FGA} explained, Mathematical Surveys and Monographs, vol. 123, American Mathematical Society, Providence, RI, 2005, pp.~235--321. \MR{2223410}

\bibitem[KS94]{KS94}
Anatole Katok and Ralf~J. Spatzier, \emph{First cohomology of {A}nosov actions of higher rank abelian groups and applications to rigidity}, Inst. Hautes \'{E}tudes Sci. Publ. Math. (1994), no.~79, 131--156. \MR{1307298}

\bibitem[KS14]{KS14-exa}
Shu Kawaguchi and Joseph~H. Silverman, \emph{Examples of dynamical degree equals arithmetic degree}, Michigan Math. J. \textbf{63} (2014), no.~1, 41--63. \MR{3189467}

\bibitem[KS16a]{KS16-abelian}
\bysame, \emph{Dynamical canonical heights for {J}ordan blocks, arithmetic degrees of orbits, and nef canonical heights on abelian varieties}, Trans. Amer. Math. Soc. \textbf{368} (2016), no.~7, 5009--5035. \MR{3456169}

\bibitem[KS16b]{KS16}
\bysame, \emph{On the dynamical and arithmetic degrees of rational self-maps of algebraic varieties}, J. Reine Angew. Math. \textbf{713} (2016), 21--48. \MR{3483624}

\bibitem[LS21]{LS21}
John Lesieutre and Matthew Satriano, \emph{Canonical heights on hyper-{K}\"{a}hler varieties and the {K}awaguchi-{S}ilverman conjecture}, Int. Math. Res. Not. IMRN (2021), no.~10, 7677--7714. \MR{4259157}

\bibitem[Mat20a]{Matsuzawa20-Adv}
Yohsuke Matsuzawa, \emph{{K}awaguchi--{S}ilverman conjecture for endomorphisms on several classes of varieties}, Adv. Math. \textbf{366} (2020), Paper No. 107086, 26 pp. \MR{4070310}

\bibitem[Mat20b]{Matsuzawa20}
\bysame, \emph{On upper bounds of arithmetic degrees}, Amer. J. Math. \textbf{142} (2020), no.~6, 1797--1820. \MR{4176545}

\bibitem[Mat23]{Mat23-survey}
\bysame, \emph{{Recent advances on Kawaguchi-Silverman conjecture}}, To appear in the proceedings for the Simons Symposia on Algebraic, Complex, and Arithmetic Dynamics (2023), 36 pp., \arxiv{2311.15489v1}.

\bibitem[MMS{\etalchar{+}}22]{MMSZZ22}
Yohsuke Matsuzawa, Sheng Meng, Takahiro Shibata, De-Qi Zhang, and Guolei Zhong, \emph{Invariant subvarieties with small dynamical degree}, Int. Math. Res. Not. IMRN (2022), no.~15, 11448--11483. \MR{4458556}

\bibitem[MSS18]{MSS18}
Yohsuke Matsuzawa, Kaoru Sano, and Takahiro Shibata, \emph{Arithmetic degrees and dynamical degrees of endomorphisms on surfaces}, Algebra Number Theory \textbf{12} (2018), no.~7, 1635--1657. \MR{3871505}

\bibitem[MY22]{MY22}
Yohsuke Matsuzawa and Shou Yoshikawa, \emph{Kawaguchi-{S}ilverman conjecture for endomorphisms on rationally connected varieties admitting an int-amplified endomorphism}, Math. Ann. \textbf{382} (2022), no.~3-4, 1681--1704. \MR{4403232}

\bibitem[MZ22]{MZ22}
Sheng Meng and De-Qi Zhang, \emph{Kawaguchi-{S}ilverman conjecture for certain surjective endomorphisms}, Doc. Math. \textbf{27} (2022), 1605--1642. \MR{4574221}

\bibitem[MZ23a]{MZ23-survey}
\bysame, \emph{{Advances in the equivariant minimal model program and their applications in complex and arithmetic dynamics}}, To appear in the Proceedings of Simons Conference (2023), 26 pp., \arxiv{2311.16369v1}.

\bibitem[MZ23b]{MZ23-ksc}
\bysame, \emph{{Structures theorems and applications of non-isomorphic surjective endomorphisms of smooth projective threefolds}}, preprint (2023), 48 pp., \arxiv{2309.07005v1}.

\bibitem[NZ10]{NZ10}
Noboru Nakayama and De-Qi Zhang, \emph{Polarized endomorphisms of complex normal varieties}, Math. Ann. \textbf{346} (2010), no.~4, 991--1018. \MR{2587100}

\bibitem[Shi19]{Shibata19}
Takahiro Shibata, \emph{Ample canonical heights for endomorphisms on projective varieties}, J. Math. Soc. Japan \textbf{71} (2019), no.~2, 599--634. \MR{3943453}

\bibitem[Sil91]{Silverman91}
Joseph~H. Silverman, \emph{Rational points on {$K3$} surfaces: a new canonical height}, Invent. Math. \textbf{105} (1991), no.~2, 347--373. \MR{1115546}

\bibitem[Sil07]{GTM241}
\bysame, \emph{The arithmetic of dynamical systems}, Graduate Texts in Mathematics, vol. 241, Springer, New York, 2007. \MR{2316407}

\bibitem[Tru20]{Tru20}
Tuyen~Trung Truong, \emph{Relative dynamical degrees of correspondences over a field of arbitrary characteristic}, J. Reine Angew. Math. \textbf{758} (2020), 139--182. \MR{4048444}

\bibitem[Yom87]{Yomdin87}
Yosef Yomdin, \emph{Volume growth and entropy}, Israel J. Math. \textbf{57} (1987), no.~3, 285--300. \MR{889979}

\bibitem[Zha06]{Zhang06}
Shou-Wu Zhang, \emph{Distributions in algebraic dynamics}, Surveys in differential geometry. {V}ol. {X}, Surv. Differ. Geom., vol.~10, Int. Press, Somerville, MA, 2006, pp.~381--430. \MR{2408228}

\bibitem[Zha09]{Zhang09}
De-Qi Zhang, \emph{A theorem of {T}its type for compact {K}\"ahler manifolds}, Invent. Math. \textbf{176} (2009), no.~3, 449--459. \MR{2501294}

\bibitem[Zha13]{Zhang13}
\bysame, \emph{Algebraic varieties with automorphism groups of maximal rank}, Math. Ann. \textbf{355} (2013), no.~1, 131--146. \MR{3004578}

\bibitem[Zha16]{Zhang16}
\bysame, \emph{{$n$}-dimensional projective varieties with the action of an abelian group of rank {$n-1$}}, Trans. Amer. Math. Soc. \textbf{368} (2016), no.~12, 8849--8872. \MR{3551591}

\bibitem[Zho22]{Zho22}
Guolei Zhong, \emph{Compact {K}\"{a}hler threefolds with the action of an abelian group of maximal rank}, Proc. Amer. Math. Soc. \textbf{150} (2022), no.~1, 55--68. \MR{4335856}

\bibitem[Zho23]{Zho23}
\bysame, \emph{Existence of the equivariant minimal model program for compact {K}\"{a}hler threefolds with the action of an abelian group of maximal rank}, Math. Nachr. \textbf{296} (2023), no.~7, 3128--3135. \MR{4626875}

\end{thebibliography}
%\bibliography{../mybib}

\end{document}